\documentclass{amsart}
\usepackage[margin=1.5in]{geometry}

\usepackage{tikz,bm, subcaption, enumerate} 
\usepackage[colorlinks=true,urlcolor=blue]{hyperref} 

\newtheorem{thm}{Theorem}[section]
\newtheorem{lem}[thm]{Lemma}
\newtheorem{prop}[thm]{Proposition}
\newtheorem{cor}[thm]{Corollary}
\newtheorem{dfn}[thm]{Definition}
\newtheorem{fact}[thm]{Fact}

\newtheorem{rmk}[thm]{Remark}

\newcommand{\refT}[1]{Theorem~\ref{#1}}

\newcommand{\refL}[1]{Lemma~\ref{#1}}
\newcommand{\refR}[1]{Remark~\ref{#1}}
\newcommand{\refD}[1]{Definition~\ref{#1}}
\newcommand{\refS}[1]{Section~\ref{#1}}
\newcommand{\refP}[1]{Proposition~\ref{#1}}

\newcommand{\refF}[1]{Figure~\ref{#1}}

\newcommand{\refand}[2]{\ref{#1} and~\ref{#2}}

\newcommand\cA{\mathcal A}
\newcommand\cB{\mathcal B}
\newcommand\cC{\mathcal C}

\newcommand\cF{\mathcal F}
\newcommand\cG{\mathcal G}

\newcommand\cI{\mathcal I}

\newcommand\cS{{\mathcal S}}
\newcommand\cT{{\mathcal T}}

\newcommand\cV{\mathcal V}

\newcommand{\E}[1]{{\mathbb E}\left[#1\right]}

\newcommand{\p}[1]{{\mathbb P}\left(#1\right)}
\newcommand{\I}[1]{{\mathbf 1}_{[#1]}}

\newcommand{\N}{\mathbb{N}}

\newcommand{\Z}{\mathbb{Z}}
\newcommand{\eps}{\varepsilon}
\newcommand{\todist}{\stackrel{\mathrm{dist}}{\longrightarrow}}
\newcommand{\eqdist}{\stackrel{\mathrm{dist}}{=}}
\newcommand{\dTV}{\mathrm{d}_{\mathrm{TV}}}

\newcommand{\Ber}[1]{\mathrm{Bernoulli}\left(#1\right)}

\newcommand{\Geo}[1]{\mathrm{Geo}\left(#1\right)}
\newcommand{\Poi}[1]{\mathrm{Poi}\left(#1\right)}
\newcommand{\Unif}[1]{\mathrm{Unif}\left(#1\right)}

\newcommand{\floor}[1]{\lfloor #1 \rfloor}
\newcommand{\ceil}[1]{\lceil #1 \rceil}

\newcommand{\n}{{(n)}}

\newcommand{\RRT}{R}

\newcommand{\RH}{\mathrm{H}}

\newcommand{\HP}{\mathcal{P}}
\newcommand{\rh}{\mathrm{h}}
\newcommand{\rT}{T}
\newcommand{\bC}{\mathbf{C}}

\newcommand{\CF}{\mathcal{C}}
\newcommand{\RD}{\mathcal{D}}
\newcommand{\rd}{\mathrm{d}}

\title[A non-increasing tree growth process]{A non-increasing tree growth process for recursive trees and applications}
\date{June 22nd 2018}
\author{Laura Eslava}


\newcommand{\Addresses}{{
  \bigskip
  \footnotesize

	\textsc{School of Mathematics, Georgia Institute of Technology,}\par\nopagebreak
	\textsc{686 Cherry Street NW, Atlanta, Georgia 30332-0160, USA} \par\nopagebreak
  \textit{E-mail address}: \texttt{laura.eslava@math.gatech.edu}\par\nopagebreak
  \textit{URL}: \texttt{http://www.people.math.gatech.edu/$\sim$leslava3/}
}}

\begin{document}

\begin{abstract}
We introduce a non-increasing tree growth process $((\rT_n,{\sigma}_n),\, n\ge 1)$, where $\rT_n$ is a rooted labeled tree on $n$ vertices and ${\sigma}_n$ is a permutation of the vertex labels. The construction of $(\rT_{n},{\sigma}_n)$ from $(\rT_{n-1},{\sigma}_{n-1})$  involves rewiring a random (possibly empty) subset of edges in $\rT_{n-1}$ towards the newly added vertex; as a consequence $\rT_{n-1} \not\subset \rT_n$ with positive probability. The key feature of the process is that the shape of $\rT_n$ has the same law as that of a random recursive tree, while the degree distribution of any given vertex is not monotonous in the process. 

We present two applications. First, while couplings between Kingman's coalescent and random recursive trees where known for any fixed $n$, this new process provides a non-standard coupling of all finite Kingman's coalescents. Second, we use the new process and the Chen-Stein method to extend the well-understood properties of degree distribution of random recursive trees to extremal-range cases. Namely, we obtain convergence rates on the number of vertices with degree at least $c\ln n$, $c\in (1,2)$, in trees with $n$ vertices. Further avenues of research are discussed. 
 \end{abstract}

\maketitle

\section{Introduction}

In a paper of 1970 \cite{NaRapoport70}, Na and Rapoport presented the problem of modeling how the structure of networks (as sociograms, communication and acquaintance networks) emerge through time. They considered two cases: \emph{ `growing'} trees and \emph{`static'} trees. 
The \emph{`growing'} model is now know as uniform attachment model and each instance is usually named (random) recursive tree. These are part of a broad class of tree growth models where vertices are sequentially added and connected to a random vertex in the current tree. On the other side, the term \emph{`static'} was motivated by the fact that this construction starts with the $n$ vertices the tree is aimed to have and $n-1$ edges are added one by one (without creating cycles). The \emph{`static'} model was an early description of what is now referred to as coalescent processes. 
The seemingly two distinct models of growth have been shown to be related for certain coalescent procedures (e.g. additive and Kingman's); that is, their resulting trees can also be constructed by a growth process \cite{ab14,LuczakWinkler04,Pitman99}. In particular, Kingman's coalescents correspond, for any fixed number of vertices $n$, to recursive trees; see \refR{rmk:K-R}.

Here we present a non-increasing tree growth process $((\rT_n,{\sigma}_n),\, n\ge 1)$ where $\rT_n$ is a rooted labeled tree on $n$ vertices and ${\sigma}_n$ is a permutation of the vertex labels. The three key features of this new growth process are:

\begin{enumerate}
\item The shape of $\rT_n$ has the same distribution as that of recursive trees (vertices are labeled uniformly at random),
\item adding edges according to the permutation $\sigma_n$ (in reverse order), recovers Kingman's coalescent,
\item there is a positive probability that $\rT_{n-1}\not\subset \rT_n$.
\end{enumerate}

Formally, we introduce the class $\RD_n$ of \emph{decorated trees} on $n$ vertices and a random mapping $\RH_n: \RD_{n-1}\to \RD_n$ such that $\RH_n(\rT_{n-1},{\sigma}_{n-1})\eqdist (\rT_n,{\sigma}_n)$ for all $n>1$. Our main result, \refT{thm:main}, states that recursively applying the mappings $\RH_n$ to the unique element in $\RD_1$ gives uniformly random decorated trees on $\RD_n$; from which the properties above are recovered. 
The fact that we can construct recursive trees in a non-increasing fashion is, to the best of our knowledge, a novel idea and it opens a wide range of further avenues of research. We discuss some of them in the last section.

We call Robin-Hood pruning to the random mapping $\RH_n$ that builds $(\rT_n,{\sigma}_n)$ from $(\rT_{n-1},{\sigma}_{n-1})$; it is the key conceptual contribution of this work and builds on the correspondence between recursive trees and Kingman's coalescent exploited in \cite{AddarioEslava15,Eslava16}. It seems that such connection had been rarely exploited, with exception to \cite{Devroye87,Pittel94}, where an equivalent construction was used by to study union-find trees and then related to recursive trees.
 
Additionally, we provide applications to high-degree vertices of recursive trees and their maximum degree. Kingman's coalescent had already been exploited by Addario-Berry and the author to describe near-maximum degrees in recursive trees, \cite{AddarioEslava15,Eslava16}. With the new procedure, we are able to extract finer information about extreme degree values in recursive trees. The main underlying technique is the Chen-Stein method for convergence rates to Poisson distributions. Informally, this method approximates the law of a sum $W$ of indicator variables, by understanding how the law of such indicator variables changes when conditioning on one of them being equal to one. In our case, the sum $W$ counts the number of vertices with high-degree. The perspective of the Robin-Hood pruning allow us to understand how the vertex-degree distributions change when we condition on one of such vertex-degrees being large.

Before we continue to precise statements of our results, we introduce basic notation that will be used throughout the paper as well as the standard construction of recursive trees. 

\subsection{Notation}

For $n\in \N$, we write $[n]=\{1,\ldots,n\}$ and $\cS_n$ for the set of permutations on $[n]$. We denote natural logarithms by $\ln(\cdot)$ and logarithms with base 2 by $\log(\cdot)$.

Given a rooted labeled tree $T=(V(T),E(T))$, write $|T|=|V(T)|$ and call $|T|$ the size of $T$. We write $\cT_n$ for the set of rooted trees $T$ with vertex set $V(T)=[n]$. 
By convention, we direct all edges toward the root $r(T)$ and write $e=uv$ for an edge with tail $u$ and head $v$. For $u\in V(T)\setminus \{r(T)\}$ we write $p_T(u)$ for the parent of $u$, that is, the unique vertex $v$ with $uv$ in $E(T)$. Finally, write $\rd_T(v)$ for the number of edges directed toward $v$ in $T$, and call $\rd_T(v)$ the degree of $v$. Note that $\rd_T(v) =\#\{u: p_T(u)=v\}$.

We say $T \in \cT_n$ is \emph{increasing} if its vertex labels increase along root-to-leaf paths; in other words, if $T \in \cT_n$ and $p_T(v) < v$ for all $v \in [n]\setminus\{r(T)\}$ (in particular, $r(T)=1$). We write $\cI_n \subset \cT_n$ for the set of increasing trees of size $n$. Using induction, it is easy to see that $|\cI_n|=(n-1)!$ for all $n$. Next, a tree growth process is a sequence $(T_n,\,  n\ge 1)$ of trees with $T_n\in \cT_n$ for each $n$. The process is increasing if $T_n\subset  T_{n+1}$ for all $n$; this implies that $T_n\in \cI_n$ for all $n$.

Recursive trees on $n$ vertices, which we denote $\RRT_n$, are usually constructed as follows. Start with $\RRT_1$ as a single node with label 1. For each $1<j\le n$, $\RRT_j$ is obtained from $\RRT_{j-1}$ by adding a new vertex $j$ and connecting it to $v_j\in [j-1]$; the choice of $v_j$ is uniformly random and independent for each $1<j\le n$. 
It is readily seen that $\RRT_n$ is a uniformly random tree in $\cI_n$. It follows that the process $(\RRT_n,\, n\ge 1)$ is a random increasing tree growth process. 
 
\subsection{The new growth process}\label{sec:main}

In what follows we extend the concept of increasing trees. If $T \in \cT_n$ and $\sigma \in \cS_n$ then $\sigma(T)$ is the tree $T' \in \cT_n$ with edges $\{\sigma(u)
\sigma(v): uv \in E(T)\}$. In words, $T'$ is obtained from $T$ by relabeling the vertices of $T$ according to the permutation $\sigma$; see \refF{transquickt} for an example. We say that $\sigma$ is an {\em stamp history} for $T$ if $\sigma(T)$ is
increasing. If $\sigma$ is an stamp history for $T$ then we say that the pair $(T,\sigma)$ is a \emph{recursively decorated tree} or \emph{decorated tree},
and that vertex $v$ has time stamp $\sigma(v)$. We denote the set of decorated trees of size $n$ by 
\[\RD_n=\{(T,\sigma):\, T\in \cT_n,\, \sigma \text{ is an stamp history of } T\}.\]

\begin{figure}
\begin{center}
\includegraphics[scale=1]{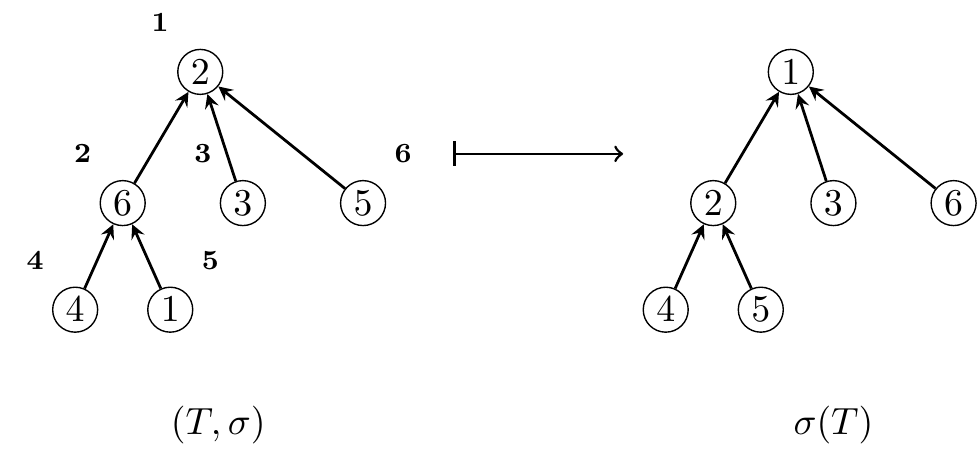}
\caption{A decorated tree $(T,\sigma)\in \RD_6$ on the left; the permutation $\sigma$ is depicted with bold numbers next to the vertices in $T$ (so for example $\sigma(1)=5$ and $\sigma(6)=2$). On the right, the increasing tree $\sigma(T)$.}\label{transquickt}
\end{center}
\end{figure}

For each $n\ge 2$, the Robin-Hood pruning $\RH_n:\RD_{n-1}\to \RD_n$ is a random mapping that can be applied to any decorated tree. The exact definition of $\RH_n$ will be given in \refS{sec:pruning}. Broadly speaking, $\RH_n(T,\sigma)$ is obtained from $(T,\sigma)$ by pruning some subtrees of $T$ and placing them as subtrees of a new vertex labeled $n$; additionally, vertex $n$ attaches to a random vertex or becomes the root of the new tree. The stamp history in $\RH_n(T,\sigma)$ is adjusted from $\sigma$ such that vertex $n$ has a uniformly random \emph{time stamp}. Heuristically, the random procedure follows a `steal from the old to give to the new' scheme; that is, once the time stamp of $n$ has been determined, vertices with an earlier time stamp have larger probability of being reattached to vertex $n$. 

The content of our main theorem says that, when the input of $\RH_n$ is uniformly random in $\RD_{n-1}$ the output is uniformly random in $\RD_n$. For the remainder of the paper, for any $n\ge 1$, the pair $(\rT_n,{\sigma}_n)$ denotes a uniformly random element in $\RD_n$. Such result boils down to carefully setting up the distribution of the random parameters involved in the Robin-Hood pruning.   

\begin{thm}\label{thm:main}
For each $n\ge 2$, the Robin-Hood pruning provides a coupling between $(\rT_{n-1},{\sigma}_{n-1})$ and $(\rT_n,{\sigma}_n)$ such that $(\rT_n,{\sigma}_n)=\RH_n(\rT_{n-1},{\sigma}_{n-1})$.
\end{thm}

Note that $|\RD_1|=1$, thus the Robin-Hood pruning can be unambiguously applied to decorated trees starting from $\RD_1$. \refT{thm:main} implies that the tree growth process $((\rT_n,{\sigma}_n),\, n\ge 1)$ given by $(\rT_n,{\sigma}_n)=\RH_n((\rT_{n-1},{\sigma}_{n-1}))$ is composed of uniformly random decorated trees, but it yields a non-increasing growth process on trees. This occurs since the rewiring may destroy some subtrees in the previous tree; see \refR{rmk:non-increase}. However, the shape of $\rT_n$ has the same law as that of $\RRT_n$; this is proven by a straightforward bijection between $\RD_n$ and $\cI_n\times \cS_n$. 

\begin{prop}\label{prop:RD-R}
For each $n\in \N$, $|\RD_n|=n!(n-1)!$ and if $(\rT_n,{\sigma}_n)\in \RD_n$ is chosen uniformly at random then ${\sigma}_n(\rT_n)\eqdist \RRT_n$ is a recursive tree of size $n$ and ${\sigma}_n$ is a uniformly chosen permutation in $\cS_n$.    
\end{prop}

\begin{proof}
By definition, if $(T,\sigma)\in \RD_n$, then $\sigma(T)\in \cI_n$. Let $\varphi:\RD_n\to \cI_n\times \cS_n$ be defined such that $\varphi(T,\sigma)=(\sigma(T), \sigma)$. For an increasing tree $T$ and $\sigma\in \cS_n$, let $T'=\sigma^{-1}(T)$ then $\varphi(T',\sigma)=(T,\sigma)$, 
it is also straightforward that $\varphi$ is injective. Therefore, $|\RD_n|=|\cI_n|\cdot |\cS_n|=n!(n-1)!$.
The result follows since bijections  preserve the uniform measure on finite probability spaces. 
\end{proof}

Growth procedures naturally couple families of trees as the size varies. For example, \refP{prop:K-DT} below shows that $(\rT_n,\sigma_n)$ is a representation of Kingman's coalescent on $[n]$; informally, the stamp history encodes the addition of edges in the coalescent. Precise definitions are given in \refS{sec:K}, for the moment it suffices to say that $\bC=(F_n,\ldots, F_1)$ denotes a Kingman's coalescent, where the $F_j$ are forests. 

\begin{prop}\label{prop:K-DT}
Let $(\rT_n,{\sigma}_n)$ be uniformly random in $\RD_n$ and $\bC=(F_n,\ldots, F_1)$ be a Kingman's coalescent. Denote $F_1=\{T_{\bC}\}$, then $T_{\bC}\eqdist \rT_n$ and the forests evolution is given by $\sigma_n$.  
\end{prop}

Typically there is no simple coupling of finite $n$-coalescent processes as $n$ varies. The first application of \refT{thm:main} is that the Robin-Hood pruning produces, given a Kingman's coalescent on $n$ vertices, a Kingman's coalescent on $n+1$ vertices. 

\begin{cor}\label{cor:main}
The tree growth process $((\rT_n,{\sigma}_n),\, n\ge 1)$, coupled as in \refT{thm:main} gives an explicit coupling of all finite Kingman's coalescents. 
\end{cor}

The proof of \refP{prop:K-DT}  is given in \refS{sec:K} and is based on previous connections between recursive trees and Kingman's coalescents; see \refR{rmk:K-R}.

\subsection{High-degree vertices in $\RRT_n$}\label{sec:Deg}

In this section we establish a phase change on the number of \emph{high-degree} vertices in recursive trees. Phase changes occurs on random structures when a class of variables undergo a transition from asymptotic normal limits to asymptotic Poisson limits. The change is marked by the mean of the variables going from infinite to bounded. In recursive trees, for example, the number of fringe trees of a given size undergoes a phase change when the size $k$ of the trees tend to infinite and $k=o(\sqrt{n})$ no longer holds \cite{Fuchs08}; similar results are given when the fringe trees are required to satisfy any given \emph{property} or \emph{pattern} \cite{ChangFuchs10,HolmgrenJanson15}. 
For an integer $0<m\le n$, let us count the number of high-degree vertices by  
\begin{align*}
Z_{m}^\n=\#\{v\in [n]: \rd_{\RRT_n}(v)\ge m \}
\end{align*}
and write $\lambda_{n,m}=\E{Z_m^\n}$. The following estimates where implicitly given in \cite{AddarioEslava15} and a proof can be found in Appendix A, \refP{upper m}. For each $c\in (0,2)$, there is $\gamma=\gamma(c)$ such that uniformly over $m<c\ln n$,  
\begin{align}\label{eq:gamma}
2^{-m+\log n}(1-o(n^{-\gamma})) =\E{Z_m^\n} \le 2^{-m+\log n}.
\end{align}
It thus follows that the phase change occurs when $m=m(n)\approx \log n$. Using a Poisson approximation together with \eqref{eq:gamma} we obtain the following phase change for the counts on high-degree vertices. 

\begin{thm}\label{thm:normal}
For each $c\in (1,\log e)$ there exists $c'\in (1,c)$ such that if $c'\ln n<m< c\ln n$, then 
\begin{align}\label{eq:normal}
\frac{Z_m^\n-\lambda_{n,m}}{\sqrt{\lambda_{n,m}}}\todist N(0,1).
\end{align}
If $m\ge \log n$, under sequences $n_j$ for which $\lambda_{n_j,m}\to \lambda$ we have that  
\begin{align*}
Z_m^{(n_j)} \todist Poi(\lambda).
\end{align*}
\end{thm}

\begin{rmk}
Counting the high-degree vertices is equivalent to count fringe trees (of all sizes) with a high-degree root. Therefore, the asymptotic normal distribution of $Z_m^\n$, \eqref{eq:normal}, follows from \cite[Corollary 1.25]{HolmgrenJanson15}; however, the computation of both the mean and variance for the renormalization of the variables is not seemingly straightforward. Nevertheless, we remark that the associated convergence rates in \refT{thm:dTV} are strong and novel.
\end{rmk}

Previous results on the profile of recursive trees consider $X_{m}^\n=\#\{v\in [n]: \rd_{\RRT_n}(v)= m \}
$, for $m< n$. For finite values of $m$, Janson established the joint limiting distribution of $(X_{m}^\n,\, m\ge 1)$ in \cite{Janson05}. Addario-Berry and the author addressed the case $m=m(n)\to \infty$, providing all the possible limiting distributions of $(X_{\floor{\log n}+k}^\n,\, k\in \Z)$ and establishing asymptotic normality for $X_m^\n$ when $m = \log  n- d$ and $d = d(n)$ slowly tends to infinity \cite{AddarioEslava15}.

\refT{thm:normal} follows from the convergence rates of the next theorem, which in turn, applies the Chen-Stein method to $Z_m^\n$. By changing the perspective of recursive trees to the distribution equivalent $\rT_n$ we can use the Robin-Hood pruning to understand how the variables $(\I_{\rd_{\rT_n}(v)\ge m},\, v\in [n])$ change when conditioning to $\rd_{\rT_n}(v)\ge m$. The details of this approach are somewhat delicate, so we defer the discussion to \refS{sec:overview}.

\begin{thm}\label{thm:dTV}
Fix $1<c'<c<2$. There are constants $\alpha=\alpha(c)\in (0,1)$ and $\beta=\beta(c')>0$ such that uniformly for $m=m(n)$ satisfying $c'\ln n< m<c\ln n$,
\begin{align*}
\dTV \left(Z_{ m}^\n, \Poi{\lambda_{n,m}}\right) &\le O(2^{-m+(1-\alpha)\log n})+O(n^{-\beta}).
\end{align*}
\end{thm}

\begin{rmk}\label{rmk:proof}
A detailed but simple track of the conditions on $\alpha$, see \refP{prop:asymp-neg}, shows that there is a non-empty interval $\cI=((1-\alpha)\log e,c)$ such that if $c'\in \cI\cap (1,2)$, then the bounds in  \refT{thm:dTV} are, in fact, tending to zero.
\end{rmk}

\begin{rmk}
The exponent $\alpha$ is determined by \emph{almost} negative correlation between pairs of vertices in $\rT_n$ (see \refP{prop:asymp-neg}), while the exponent $\beta$ depends on an auxiliary coupling based on the Robin-Hood pruning (see \refP{prop:coup-neg}). We believe that the constraint on $c'>1$ could be relaxed by obtaining uniform bounds on $\p{\rd_{\RRT_n}(i)=m}$ rather than $\p{\rd_{\RRT_n}(i)\ge m}$. 
\end{rmk}

Finally, consider now the maximum degree $\Delta_n$ of a recursive tree $\RRT_n$. Note that $Z_m^\n>0$ if and only if $\Delta_n\ge m$. Therefore, having $\E{Z^\n_{\log n}}\approx 1$ indicates $\Delta_n\approx \log n$. In fact, Devroye and Lu showed that $\Delta_n/\log n\to 1$ a.s. \cite{DevroyeLu95}. The first tail bounds on $\Delta_n$ where obtained for $\p{\Delta_n<\lfloor \log n\rfloor +i}$ with $i\in \Z$ using singularity analysis of generating functions \cite{GohSchmutz02}. The relation between recursive trees and Kingman's coalescent provided simpler proofs to such results, extending it also to $i<2\ln n-\log n$ \cite{AddarioEslava15}. The bounds in \refT{thm:dTV} yield broader, tighter bounds. 

\begin{cor}\label{cor:Delta}
There exists $C>0$ such that uniformly over $0<i=i(n)<\log e \ln \ln n-C$,
\[\p{\Delta_n<\lfloor \log n\rfloor -i}=\exp\{-2^{i+\eps_n}\}(1+o(1)),\]
where $\eps_n=\log n-\floor{\log n}$. 
\end{cor}

The maximum of i.i.d. random variables is, under rather general conditions, distributed in the limit as the Gumbel (or double-exponential) distribution \cite{Gumbel35}; however lattice distributions are excluded from this regime. Addressing the case of integer-valued variables, Anderson gives sufficient conditions under which the Gumbel distribution serves as an approximation for their maximum \cite{Anderson70}; among those is the geometric distribution.
Now, when we randomize the labels in $\RRT_n$ (e.g. using the tree $\rT_n$ instead), vertex degrees become exchangeable and their limiting distributions are geometric. Although, the degrees of $\rT_n$ are not independent, their correlations are weak and the  Gumbel-type approximation still arises for the distribution of $\Delta_n$. Goh and Schmutz provide an alternative heuristic based on the fact that $\rd_{\RRT_n}(i)$, with $i\to \infty$ slowly, is asymptotically normal \cite{GohSchmutz02}.

\subsection*{Outline}
The paper is divided into two parts. First, we discuss more on the connection between recursive trees, Kingman's coalescents and other tree models in \refS{sec:K}. The precise definition of the Robin-Hood pruning $\RH_n$ together with the proof of \refT{thm:main} is given in \refS{sec:pruning}. Second, the results on high-degree vertices of recursive trees use the Chen-Stein method and the Robin-Hood pruning in a non-trivial way. An overview on how we use the Chen-Stein method is given \refS{sec:overview}. Assuming the existence of an auxiliary coupling (\refP{prop:coup-neg}), we complete the proofs concerning high-degree vertices in \refS{sec:DegProofs}. The auxiliary coupling, based on the Robin-Hood pruning, is presented in \refS{sec:DegCrux}. 
And finally, \refS{sec:Conclusions} discusses further avenues of research. 

\section{Kingman's coalescents and recursive trees: distinct representations}\label{sec:K}

Discrete coalescents are processes on partitions of $[n]$ that can be represented with different tree structures. On can encode the coalescent using an $n$-chain: a sequence of forests where, at all times there are $n$ vertices (or \emph{elements}), and $n-1$ edges are added one by one until a tree is formed. However, there is an more traditional construction using binary search trees (BST) where internal nodes correspond to merges and only external nodes correspond to \emph{elements} of the coalescent. In the next section we introduce the representation used in this paper and prove \refP{prop:K-DT}. Following that, we discuss the well-know bijection between BST's and recursive trees and the difference between the two coalescent representations. In addition, we explain the difference between the Yule-Harding model of phylogenetic trees and its uniform model, and highlight the importance of clarifying both the rules applied to the mergings in coalescent processes and their representation as trees.

\subsection{Recursive trees perspective}

Na and Rapoport loosely described this process as the construction of \emph{`static'} trees with $n$ vertices \cite{NaRapoport70}:  

\begin{quote}
``Initially, single elements
move about at random.
Each collision forms a couple. A collision of a
couple with a single element forms a triple, a collision of an $s$-tuple with a
$t$-tuple forms an $(s + t)$-tuple, and so on. At each collision a link is
established between an element of one $X$-tuple and an element of another,
the links being rigid so that the elements of the same $k$-tuple cannot collide.
The process goes on until the entire set of $n$ elements has been joined into
an $n$-tuple.''
\end{quote}

By changing the rule on how to link the elements on the tuples, we obtain distinct coalescent distributions. In the description in \cite{NaRapoport70} there are no restrictions on the elements allowed to be linked during the coalescent\footnote{Unfortunately, it was incorrectly presumed in \cite{NaRapoport70} that `static' trees build uniformly random unrooted labeled trees.}. In fact, the discrete multiplicative coalescent arises when any possible link is chosen uniformly at random. It is associated with Kruskal's algorithm for the minimum weighted spanning tree problem \cite{ab14}. 

Kingman's coalescent is characterized by the property that the merging probability of any pair of components is independent of the components' sizes. The representation used in this paper uses that, at each time the representative of each `tuple' is the current root of the tree and it is closely related to the `union-find' algorithm used in computer science (see e.g. \cite{Standish80}). A formal description follows. 

\begin{figure}
\begin{center}
\includegraphics[scale=.8]{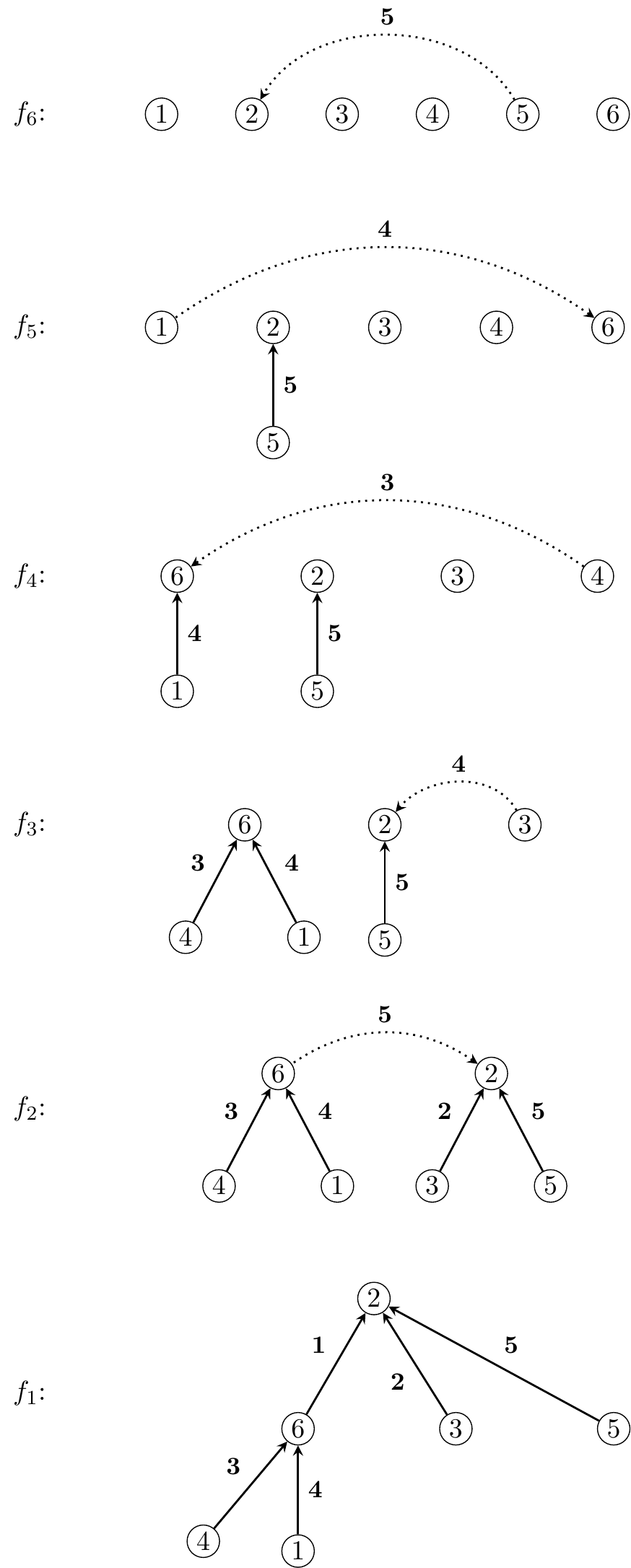}
\caption{An example of an $n$-chain with $n=6$. The edge labelling $\rho_n$ is presented with numbers in bold.}\label{chain}
\end{center}
\end{figure}

A forest $f$ is a set of trees with pairwise disjoint vertex sets. Denote by $V(f)$ and $E(f)$, respectively, the union of the vertex and edge sets of the trees in $f$. For $n\ge 1$, an $n$-chain is a sequence $C=(f_n,\ldots,f_1)$ of elements of $\cF_n=\{f:V(f)=[n]\}$ such that, for $1< i\le n$, $f_{i-1}$ is obtained from $f_i$ by adding a directed edge between the roots of some pair of trees in $f_i$. In particular, $f_n$ consists of $n$ one-vertex trees and $f_1$ consists of a single tree on $n$ vertices denoted by $T_C$. For an example see \refF{chain}.

Next we introduce the necessary notation to define Kingman's coalescent using $n$-chains. For an $n$-chain $(f_n, \ldots, f_1)$ and $1\le i\le n$, list the trees of $f_i$ in increasing order of their smallest-labeled vertex as $T_1^{(i)}, \ldots, T_{i}^{(i)}$. Independently for each $1< i\le n$ let $\{a_i,b_i\}\subset \{\{a,b\}: 1\le a<b\le i\}$ be uniformly chosen at random; in addition, let $\xi_i$ be independent Bernoulli random variables with mean $1/2$. 

\begin{dfn}\label{dfn:Kingman}
Kingman's $n$-coalescent is defined as $\bC=(F_n,\ldots,F_1)$ constructed as follows.
For $1< i\le n$, $F_{i-1}$ is obtained from $F_i$ by adding an edge between $r(T_{a_i}^{(i)})$ and $r(T_{b_i}^{(i)})$. If $\xi_i=1$ then direct the edge towards $r(T_{a_i}^{(i)})$; otherwise direct it towards $r(T_{b_i}^{(i)})$. The forest $F_{i-1}$ consists of the new tree and the remaining $i-2$ unaltered trees from $F_i$.
\end{dfn}

In other words, if  $\bC=(F_n,\ldots,F_1)$ is a Kingman's coalescent, then each of the trees of $F_i$ correspond to a set of coalesced elements after $n-i+1$ steps of the process. At each step, two sets (represented by their roots) coalesce and a new representative is chosen uniformly at random. 

To link $n$-chains with decorated trees, we first define a natural edge labeling that tracks the number of \emph{trees left} in the forest when a give edge comes along. Fix $C=(f_n,\ldots, f_1)$, for each $e\in E(T_C)$, let 
\[\rho_C(e)=\max\{i\in [n-1]:\, e\in E(f_i)\}.\]
We next define a vertex labeling $\sigma_C:V(T_C)\to [n]$. Let 
$\sigma_C(r(T_C))=1$, and for each $uv\in E(t_C)$,  let 
\[\sigma_C(u)=\rho_C(uv)+1.\]

The following proposition shows that the pair $(T_C,\sigma_C)\in \RD_n$ contains all the information to recover the original $n$-chain $C$; in other words, if $\CF_n$ denotes the set of $n$-chains, then $\RD_n$ and $\CF_n$ are in bijection. 

\begin{prop}\label{prop:RD-CF}
Let $\Upsilon:\CF_n \to \RD_n$ be defined as follows. For an $n$-chain $C=(f_n,\ldots, f_1)$, let $\Upsilon(C)=(T_C,\sigma_C)$. Then $\Upsilon$ is a bijection.  
\end{prop}

\begin{proof}
First, we show that $\CF_n$ and $\RD_n$ have the same cardinality. To count the number of $n$-chains, consider constructing $(f_n,\ldots, f_1)$ by deciding which edge to add from $f_{k}$ to $f_{k-1}$. Since there are $k$ trees in $f_k$, when we have chosen $(f_n,\ldots, f_k)$, there are $k(k-1)$ possible directed edges to add. Therefore, $|\CF_n|=n!(n-1)!$.

Next, let $C=(f_n,\ldots,f_1)$ be an $n$-chain. For each $1\le i< n$, the new edge in $f_i$ joins the roots of two trees in $f_{i+1}$ and is directed towards the root of the resulting tree. Thus, the labels $\{\rho_C(e),\, e\in E(T_C)\}$ decrease along all paths in $T_C$ towards the root $r(T_C)$. Consequently, the labels $\{\sigma_C(v),\, v\in [n]\}$ are, indeed, an stamp history of $T_C$. It follows that $\Upsilon$ is well defined. 

Finally, let $C=(f_n,\ldots, f_1)$, $C'=(f'_n, \ldots, f'_1)$ be distinct $n$-chains and write	 $k=\min\{i:\, f_i\neq f'_i \}$. If $k=1$ then $T_C\neq T_{C'}$ and clearly, $\Upsilon(C)\neq \Upsilon(C')$. Otherwise $T_C=T_{C'}$, $f_{k-1}=f'_{k-1}$ and the (unique) edges $e\in E(f_{k-1})\setminus E(f_k)$ and $e'\in E(f'_{k-1})\setminus E(f'_k)$ are distinct. It follows that $e=uv\in f'_k$ and so $\sigma_C(u)=k> \sigma_{C' }(u)$. This shows that $\Upsilon$ is injective, and so $\Upsilon$ is a bijection between $\CF_n$ and $\RD_n$. 
\end{proof}

Using the bijection of \refP{prop:RD-CF}, it follows that \refP{prop:K-DT} boils down to showing that $\bC$ is uniformly random in $\CF_n$. 

\begin{proof}[Proof of \refP{prop:K-DT}]
Let $\bC=(F_n,\ldots, F_1)$ be a Kingman's coalescent. For any fixed $n$-chain $(f_n,\ldots, f_1)\in \CF_n$, 
\[\p{(F_n,\ldots,F_1)=(f_n,\ldots, f_1)}
=\prod_{k=1}^{n-1} \p{F_k=f_k|(F_n,\ldots,F_{k+1})=(f_n,\ldots, f_{k+1})}. \]
Among the $k(k+1)$ possible oriented edges connecting roots of $f_{k+1}$, exactly one of them can be added to $f_{k+1}$ to yield $f_k$. Thus, regardless of the sequence $(f_n,\ldots, f_1)$,
\[\p{(F_n,\ldots,F_1)=(f_n,\ldots, f_1)}=[(n-1)!n!]^{-1}.\]
Recall $F_1=\{T_{\bC}\}$. By \refP{prop:RD-CF}, $(T_{\bC},{\sigma}_{\bC}) \in \RD_n$ and it has a uniform distribution, since the bijection preserves the uniform measure of $\bC$. Finally, by \refP{prop:RD-CF}, it follows that $T_{\bC}\eqdist \rT_n$. The evolution of the forests is given by $\rho_\bC$; equivalently by $\sigma_{\bC}$. 
\end{proof}

\begin{rmk}\label{rmk:K-R}
It follows from Propositions~\refand{prop:RD-R}{prop:K-DT} that, for any fixed $n$ and up to relabeling of vertices, Kingman's coalescent correspond to recursive trees. See also \cite{ab14,AddarioEslava15} for direct proofs of this fact. 
\end{rmk}

\subsection{The binary search tree connection}

Binary search trees have been related to both recursive trees and phylogenetic trees. In this section we briefly discuss these connections and compare them with Kingman's coalescent. Let $\cB_n$ be the set containing all plane, rooted, unlabeled binary trees with $n$ external nodes. Trees in $\cB_n$ distinguish between left and right subtrees of any given internal vertex. It can be shown that the sizes $|\cB_n|=3\cdot5\cdots (2n-3)$ are given by the Catalan numbers. 

Binary search trees are the tree representation of the sorting algorithm Quicksort. Simply described, for each $n\ge 1$, the quicksort algorithm takes a permutation $\sigma\in [n]$ and constructs (step by step) a binary tree with internal vertex labels on $[n]$ as follows. The root is $\sigma(1)$ and vertices $\sigma(2),\ldots, \sigma(n)$ are sequentially added so that the final tree satisfies the following property: for any internal node $j$, all nodes on its left subtree are smaller than $j$ and all nodes on its right tree are larger than $j$. It follows that, given a shape of a binary tree $B\in \cB_{n+1}$, there is exactly one way to label internal vertices. 

There are $n!$ distinct permutations as input for the quicksort algorithm. Devroye introduced the representation of the binary search tree (process) using time stamps which record all the insertion process. Using this representation, the rotation correspondence maps (one-to-one) recursive trees (on $n-1$ vertices) and binary search trees (with $n$ external vertices). For a thorough description of the correspondence see \cite[Section 2, Figures 1-2]{HolmgrenJanson15}. 

On the other side, phylogenetic trees on $n$ species are represented by elements in $\cB_n$. In this case, species are assumed to have a common ancestor (the root), internal nodes are also ancestors and the time elapsed between differentiation of species, length of the branches, is omitted. Two common distributions on phylogenetic trees are the uniform one, known as the Catalan model, and the Yule-Harding model. The latter is a process that constructs trees starting from the root, by branching a uniformly random external node and replacing it with a cherry (an internal node with two external nodes). Clearly, this construction corresponds one-to-one to the quicksort algorithm. We remark that the Yule-Harding process does not yield uniform phylogenetic trees (as neither the BST is uniform in $\cB_n$). For further discussion between the two models, see e.g. \cite[Section 3]{BlumFrancoisJanson06}.
 
It has been presumed that Kingman's coalescent is the bottom's up construction of the Yule-Harding model, see e.g. \cite{ChangFuchs10}; however, such correspondence has to be done carefully, as merges in principle are not bound to satisfy planarity constraints. As we can see through the bijections and $n!$-to-1 mappings from Propositions \refand{prop:RD-CF}{prop:RD-R}, there should be a correspondence between BST's with time stamps and Kingman's coalescent. 

The construction of Kingman's coalescent as a binary tree in $\cB_n$ with time stamps is the following. Using the same random variables used in \refD{dfn:Kingman}, add an internal node connecting the two roots of the merging trees, while the coin flip indicates which of the trees is left child of the new internal vertex; the time stamps indicate the (reversed) order of addition of internal nodes. Note that in this construction, the symmetry breaking of the coin flip is still necessary.

Conversely, we describe how to interpret time stamps of a BST as the merging history of a Kingman's coalescent. To do so, we have to label external nodes uniformly at random (so that there are a total of $n!(n-1)!$ different processes). Now, the role of internal vertices is as follows. At step $k\in [n-1]$, the two set of external vertices in each of the subtrees of the vertex with time stamp $n-k$ are the subsets to be merged in the coalescent. 

Kingman's coalescent has uniform distribution when considering all possible merging histories with elements labeled exchangeably. However, considering only the final tree (either in $\cT_n$ or $\cB_n$) yields a non-uniform distribution: there are $n!(n-1)!$ total ways to merge the subtrees (if we use the symmetry breaking at each merging), but there are only $|\cT_n|=n^{n-1}$ and $|\cB_n|=3\cdot5\cdots (2n-3)$ different rooted, labeled trees and phylogenetic trees, respectively. 

\section{The Robin-Hood pruning} \label{sec:pruning}

The Robin-Hood pruning $\RH_n:\RD_{n-1}\to \RD_n$ is a random procedure based on randomizing the parameters  of a deterministic mapping $\rh_n:\RD_{n-1}\times \HP_{n}\to \RD_n$ where the set $\HP_{n}$ defines all possible ways to prune a decorated tree on $n-1$ vertices. The distribution on $\HP_n$ is tailored so that the Robin-Hood pruning, in fact, yields a coupling of $((\rT_n,\sigma_n);\, n\ge 1)$.  

First we introduce the necessary notation to define $\HP_n$, the deterministic pruning $\rh_n$ and verify that, indeed, the mapping $\rh_n$ is well defined. We then continue to define the distribution on $\HP_n$ used in defining the Robin-Hood pruning (\refD{dfn:RH-set}). The proof of \refT{thm:main} requires us to characterize the properties of the uniform distribution in decorated trees. 
For the characterization in \refL{lem:charac} and the proof of \refT{thm:main}, we underline the difference between deterministic elements $(T,\pi)$ of $\RD_n$ and random elements $(\bm{\rT},\bm{\sigma})$ using bold notation; the distribution of $(\bm{\rT},\bm{\sigma})$ is not given a priori. 

\subsection{A deterministic process}
Informally, we define all possible ways to prune a decorated tree on $n-1$ vertices using three parameters $(k,l,x)\in \HP_n$: the time stamp $k$ of the new vertex, its point of attachment $l$ given by time stamp, and the vertices to be rewired encoded by time stamp in the sequence $x=(x_1,\ldots, x_{n-1})$. Once a stamp history is given to a tree $T$, $\cV_n$ contains the vertices to be pruned and rewired towards the new vertex $n$. 

We now proceed to precise definitions. Let $n\ge 2$ and set
\[\HP_n=
\{(k,l,x):\, 1\le l<k\le n,\, x\in \{0,1\}^{n-1}\} \cup
\{(1,0,x):\, x\in \{0,1\}^{n-1},\, x_1=1\} ; \]
additionally, for $(k,l,x)\in \HP_{n}$ and a permutation $\sigma\in \cS_{n-1}$, let 
\[\cV_n(k,l,x,\sigma)=\cV_n(k,x,\sigma)
=\{v\in [n-1]:\, x_{\sigma(v)}=1,\, \sigma(v)\ge k\}.\] 

\begin{rmk}
The definition of $\HP_n$ is such that $\sigma^{-1}(1)\in \cV_n$ if and only if $k=1$.
\end{rmk}

The following deterministic \emph{pruning} is illustrated in  \refF{pruning}.
\begin{dfn}
Fix $n\ge 2$, $(T,\sigma)\in \RD_{n-1}$ and $(k,l,x)\in \HP_{n}$ . We define $(T',{\sigma}')$ and set 
 \[\rh_n((T,\sigma),(k,l,x))=(T',\sigma')\]
as follows. First, let $\cV=\cV_{n}(k,x,\sigma)$ and construct $T'$ from $T$: For each $v\in \cV\setminus \{r(T)\}$, replace the edge $vp_T(v)$ with an edge connecting $v$ to a new vertex labeled $n$. Now, if $k=1$ then  attach $r(T)$ to $n$; 
otherwise, attach vertex $n$ to $\sigma^{-1}(l)$. In other words, the edges of $T'$ are given by 

\[E(T')=
\begin{cases}
\qquad \quad (E(T)\cup \{vn;\, v\in \cV\}) \setminus \{ vp_T(v);\, v\in \cV\}  &\text{ if } k=1,\\
 \{n\sigma^{-1}(l)\}\cup (E(T)\cup \{vn;\, v\in \cV\}) \setminus \{ vp_T(v);\, v\in \cV\}&\text{ if } k>1.\end{cases}\]

Second, let ${\sigma}':[n]\to [n]$ be defined by ${\sigma}'(n)=k$ and for $v<n$, 
\[{\sigma}'(v)=\sigma(v)+\I{\sigma(v)\ge k}.\]
\end{dfn}

\begin{lem}\label{lem:T-rh}
For any $n\ge 2$, $\rh_n:\RD_{n-1}\times \HP_n\to \RD_n$ is well defined. That is, for any $(T,\sigma)\in \RD_{n-1}$ and $(k,l,x)\in \cC_{n}$, 
\[\rh_n((T,\sigma),(k,l,x))\in \RD_{n}.\]
\end{lem}

\begin{proof}
Write $\rh_n((T,\sigma),(k,l,x))=(T',\sigma')$. When $k=1$, it is clear that $T'$ is a tree. When $k>1$, let $w=\sigma^{-1}(l)$ be the parent of $n$ in $T'$ and let $(w=v_1,\ldots, v_j=r(T))$ be the path from $w$ to the root of $T$. Since $\sigma$ is a stamp history of $T$, 
$l=\sigma(v_1)>\sigma(v_2)>\cdots>\sigma (v_j)=1$;
moreover, $l<k$. It follows that $v_i\notin \cV(k,l,x,\sigma)$ for all $i\in [j]$ and consequently, no edges in the path from $n$ to the root in $T'$ closes a cycle by connecting to $n$. 

Now, we show that ${\sigma}'$ is a stamp history for $T'$. It is clear that ${\sigma}'$ is a permutation of $[n]$, so it suffices to prove that  ${\sigma}'(v)>{\sigma}'(p_{T'}(v))$, for all $v\in V(T)\setminus \{r(T')\}$. 
First, for vertices $v$ with $p_{T'}(v)=n$ we have $\sigma(v)\ge k$ and consequently 
${\sigma}'(v)=\sigma(v)+1>k={\sigma}'(n)$. 

Second, consider $v,w< n$ with $p_{T'}(v)=w$. It follows that $vw\in E(T)$ and thus  $\sigma(v)>\sigma(w)$. Consequently, $\I{\sigma(v)\ge k}\ge \I{\sigma(w)\ge k}$ and so ${\sigma}'(v)>{\sigma}'(w)$.
The last case occurs when $k>1$ and $p_{T'}(n)=w=\sigma^{-1}(l)$. We then have 
${\sigma}'(n)=k>l=\sigma(w)={\sigma}'(w)$.
\end{proof}

\begin{rmk}\label{rmk:non-increase}
Whenever $(k,l,x)\in \HP_n$ has $x_j=1$ for some $j\ge k$, setting $(T',\sigma')=\rh_n((T,\sigma),(k,l,x))$ and $v=\sigma^{-1}(j)$ yields $n=p_{T'}(v)\neq p_T(v)\in [n-1]$. This implies that $T\not\subset T'$. 
\end{rmk}

\begin{center}
\begin{figure}
\begin{subfigure}{\textwidth}
\centering
\includegraphics[scale=.7]{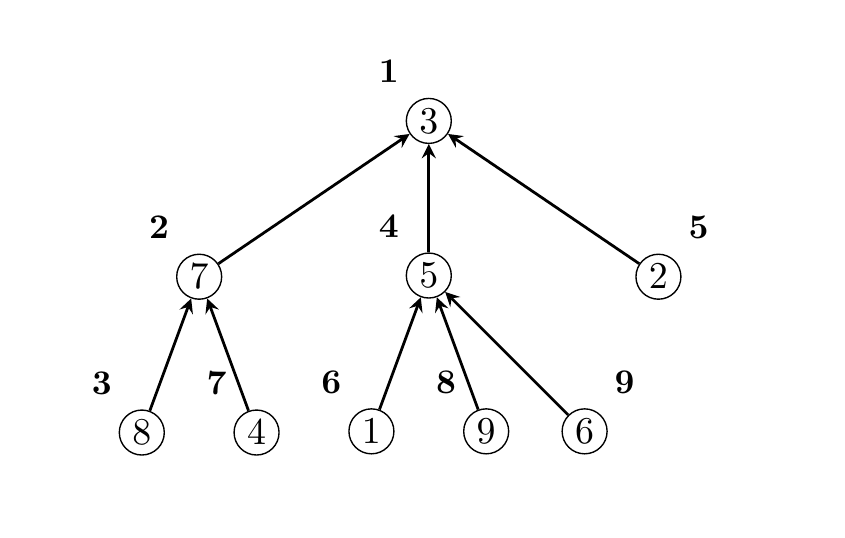}
\caption{A tree $(T,\sigma)$ in $\RD_9$. The permutation $\sigma$ is depicted with bold numbers.}
\end{subfigure}
\begin{subfigure}{\textwidth}
\centering
\includegraphics[scale=.65]{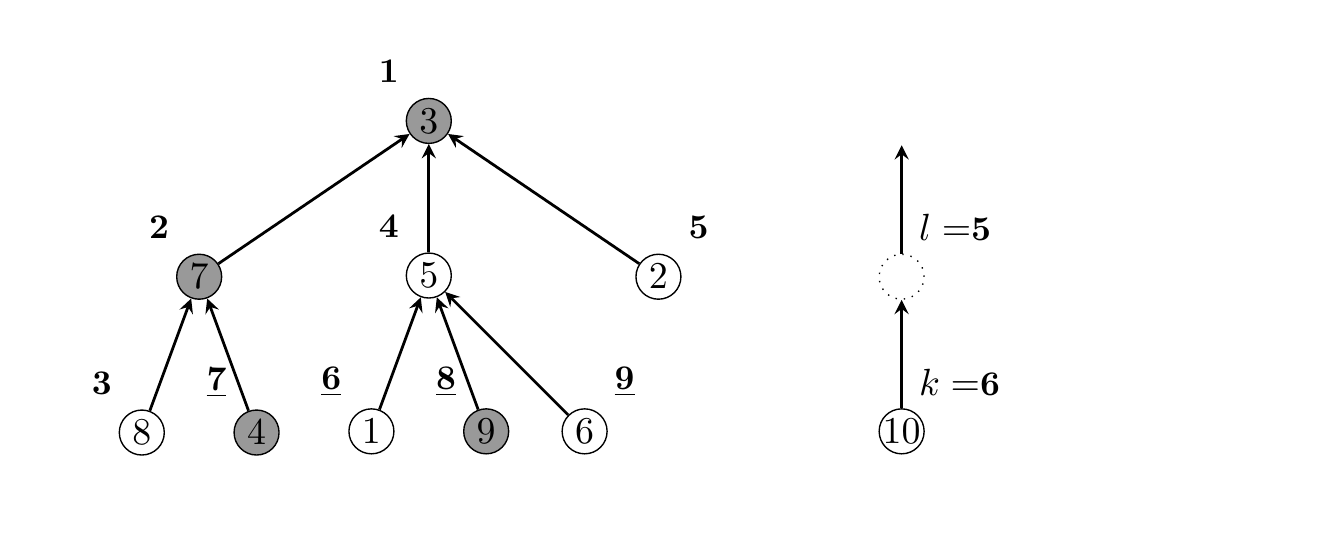}
\caption{Vertices in gray satisfy $X_{\sigma(v)}=1$ and underlined are time stamps $\sigma(i)\ge k$.}
\end{subfigure}
\begin{subfigure}{\textwidth}
\centering
\includegraphics[scale=.65]{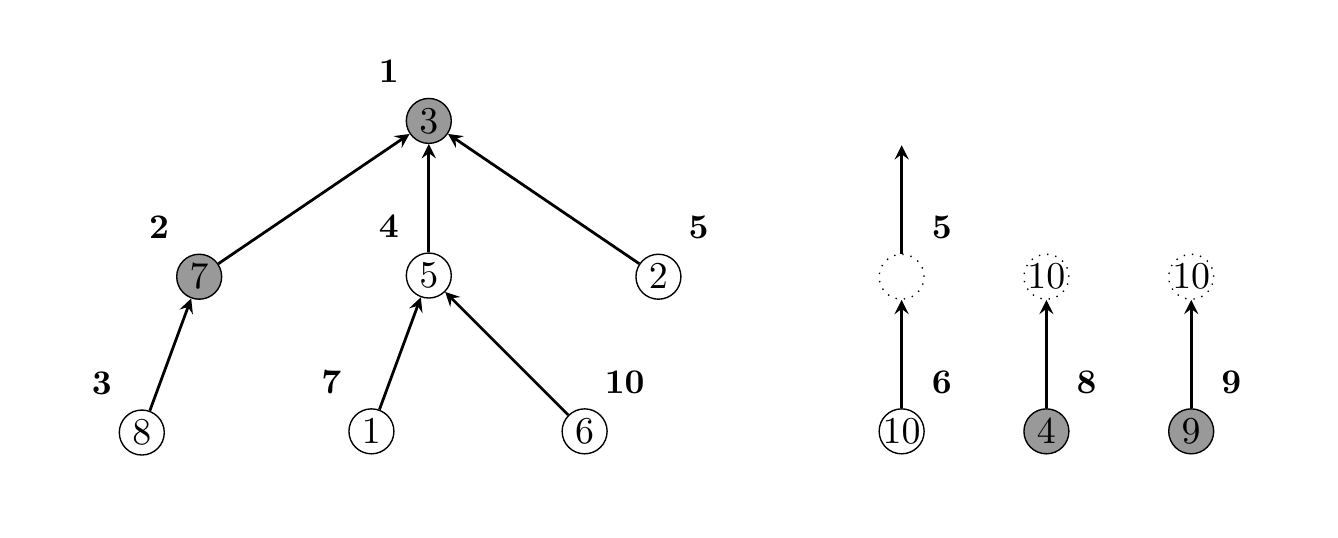}
\caption{Nodes $i$ with $\sigma(i)\ge k$ and $x_i=1$ have been pruned and time stamps have been adjusted.}
\end{subfigure}
\begin{subfigure}{\textwidth}
\centering
\includegraphics[scale=.7]{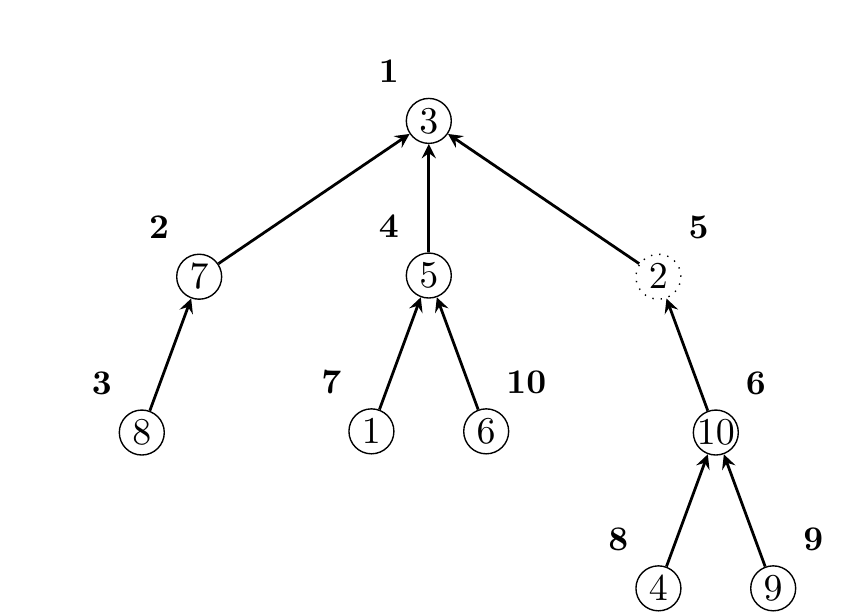}
\caption{The resulting tree $\rh_{10}((T,\sigma),(k,l,x))\in \RD_{10}$.}
\end{subfigure}
\caption{An example of the Robin-Hood pruning for $(T,\sigma)$ with $k=6,l=5$ and $x_1=x_2=x_7=x_8=1$; all other $x_i=0$.}\label{pruning}
\end{figure}
\end{center}

\subsection{The random process}
The $\RH_n$-set is a sample of  $\HP_n$ according to the following distribution. 

\begin{dfn}\label{dfn:RH-set}
Fix $n\ge 1$. Let $K\eqdist \Unif{1,2,\ldots, n}$; if $K=1$ let $L=0$, and if $K>1$ let $L= \Unif{1,2,\ldots K-1}$. Independently, let $X=(X_1,\ldots, X_{n-1})$ where $X_i\eqdist \Ber{1/i}$ are independent variables. An $\RH_n$-set is a triple of random variables with the same law as $(K,L,X)\in \HP_n$.
\end{dfn}

We are ready to define the Robin-Hood pruning. For each $n\ge 2$, let $(K,L,X)\in \HP_{n}$ be an $\RH_{n}$-set and define
\[\RH_n(T,\sigma)=\rh_n((T,\sigma),(K,L,X)).\]

The law of $\RH_n(T,\sigma)$ depends on the initial input $(T,\sigma)$; however, the distribution of the $\RH_{n}$-set is tailored so that $\RH_n(\rT_{n-1},\sigma_{n-1})$ preserves the uniform measure on decorated trees.
In order to prove \refT{thm:main}, we start with a characterization of $(\rT_n,\sigma_n)$.

\begin{lem}\label{lem:charac}
Let $n\ge 1$ be an integer. A random decorated tree $(\bm{\rT},\bm{\sigma})\in \RD_n$ is uniformly random if and only if the following properties are satisfied.
\begin{enumerate}[i)]
\item The permutation $\bm{\sigma}$ is uniformly random on $\cS_n$.
\item Conditionally given $\bm{\sigma}$, the vertices 
$(p_{\bm{\sigma}(\bm{\rT})}(\bm{\sigma}^{-1}(v)),\, v\in V(\bm{\rT})\setminus \{r(\bm{\rT})\} )$
 are independent. 
\item For all vertices $v,w\in [n]$ and indices $i,j \in [n]$, 
\begin{align}\label{uniform}
\p{p_{\bm{\rT}}(v)=w,\, \bm{\sigma}(v)=j,\, \bm{\sigma}(w)=i}
=\frac{1}{n(n-1)(j-1)}\I{j>i}. 
\end{align}
\end{enumerate}
\end{lem}

\begin{proof}
Let $(\bm{\rT},\bm{\sigma})=(\rT_n,\sigma_n)$ be uniformly random on $\RD_n$. Condition $i)$ follows directly from \refP{prop:RD-R} which states that $\bm{\sigma}$ is a uniformly random permutation and $\bm{\sigma}(\bm{\rT})$ has the law of a recursive tree $\RRT_n$. In addition, \refP{prop:RD-R} implies 
\begin{equation*}
(p_{\bm{\sigma}(\bm{\rT})}(\bm{\sigma}^{-1}(v)),\, v\in V(\bm{\rT})\setminus \{r(\bm{\rT})\} )\eqdist \{p_{\RRT_n}(j),\, 1<j\le n\},
\end{equation*}
from which conditions $ii)$ and $iii)$ immediately follow: Parents in recursive trees are chosen independently for each of the vertices, and for all $v,w,i,j\in [n]$, 
\begin{align*}
\p{p_{\bm{\rT}}(v)=w,\, \bm{\sigma}(v)=j,\, \bm{\sigma}(w)=i}
&=\frac{1}{n(n-1)} \p{p_{\bm{\rT}}(v)=w\,|\, \bm{\sigma}(v)=j,\, \bm{\sigma}(w)=i} \\
&=\frac{1}{n(n-1)} \p{p_{\bm{\sigma}(\bm{\rT})}(j)=i} \\
&=\frac{1}{n(n-1)(j-1)}\I{j>i}. 
\end{align*}

Now consider a random decorated tree $(\bm{\rT},\bm{\sigma})\in \RD_n$ satisfying  conditions $i)$-$iii)$. Fix a decorated tree $(T,\pi)\in \RD_n$, and for $v\in V(T)\setminus \{r(T)\}$, let $w_v=p_T(v)$. 
Condition $ii)$ on the conditional independence of parents gives, for $v\neq r(T)$,
\begin{align*}
\p{p_{\bm{\rT}}(v)=w_v|\bm{\sigma}=\pi} 
=&\p{p_{\bm{\sigma}(\bm{\rT})}(\bm{\sigma}(v))=\pi(w_v)|\, \bm{\sigma}(v)=\pi(v),\, \bm{\sigma}(w_v)=\pi(w_v)}\\
=&\frac{\p{p_{\bm{\rT}}(v)=w_v,\, \bm{\sigma}(v)=\pi(v),\, \bm{\sigma}(w_v)=\pi(w_v)}}{\p{\bm{\sigma}(v)=\pi(v),\, \bm{\sigma}(w_v)=\pi(w_v)}}
\end{align*} 
Using that $\pi$ is an stamp history for $T$, so $\pi(v)>\pi(w_v)$, and that 
$\bm{\sigma}$ is uniformly random, it follows from \eqref{uniform} that
\begin{align}\label{frac}
\p{p_{\bm{\rT}}(v)=w_v|\bm{\sigma}=\pi} =\frac{1}{\pi(v)-1}.
\end{align} 

Any increasing tree $T'\in \cI_n$  is determined by the set of parents $\{p_{T'}(v),\, 1<v\le n\}$. Using that $\pi(T)\in \cI_n$ and the conditional independence from condition $ii)$ we get 
\begin{align*}
\p{\bm{\sigma}(\bm{\rT})=\pi(T)\,|\, \bm{\sigma}=\pi}
&=\quad \p{p_{\bm{\sigma}(\bm{\rT)}}(j)=p_{\pi(T)}(j),\, 1<j\le n\,|\, \bm{\sigma}=\pi}\\
&=\quad \p{p_{\bm{\rT}}(v)=p_{T}(v),\, v\in V(t)\setminus \{r(T)\}\,|\, \bm{\sigma}=\pi}\\
&=\prod_{v\in V(T)\setminus r(T)} \p{p_{\bm{\rT}}(v)=p_{T}(v)\,|\, \bm{\sigma}=\pi}\\
&=\quad [(n-1)!]^{-1};
\end{align*}
the last equality holds by \eqref{frac} and the fact that $\{\pi(v),\, v\in  V(T)\setminus \{r(T)\}\}=\{2,\ldots, n\}$. Finally, using the equation above and that $\bm{\sigma}$ is uniformly random, we have
\begin{align*}
\p{(\bm{\rT},\bm{\sigma})=(T,\pi)}
&=\p{\bm{\sigma}(\bm{\rT})=\pi(T)\,|\, \bm{\sigma}=\pi}\p{\bm{\sigma}=\pi}\\
&=\frac{1}{n!}\p{\bm{\sigma}(\bm{\rT})=\pi(T)\,|\, \bm{\sigma}=\pi}
= [n!(n-1)!]^{-1}.
\end{align*}
This holds regardless of the choice of $(T,\pi)$, so $(\bm{\rT},\bm{\sigma})$ is uniformly random in $\RD_n$.

\end{proof}

We are now ready to prove \refT{thm:main}.

\begin{proof}[Proof of \refT{thm:main}]
Let $(\rT_{n-1},\sigma_{n-1})\in \RD_{n-1}$ be a uniformly random decorated tree. Let $(K,L,X)$ be an $\RH_{n}$-set and let $(\bm{\rT}, \bm{\sigma})=\rh((\rT_{n-1},\sigma_{n-1}),(K,L,X))$. It suffices to show that $(\bm{\rT},\bm{\sigma})$ satisfies the properties in \refL{lem:charac}. 

First, condition $i)$ follows from the construction of $\bm{\sigma}$ and the distributions of both $K$ and $\sigma_{n-1}$. Second, once conditioning on $\bm{\sigma}$, which is equivalent to conditioning on both $\sigma_{n-1}$ and $K$, we get 
\begin{align*}
\{p_{\bm{\rT}}(v),\, v\in V(\bm{\rT})\setminus \{r(\bm{\rT})\}\}
=&\{p_{\bm{\rT}}(v),\, 1<\bm{\sigma}(v)<\bm{\sigma}(n)\}\\
& \cup \{p_{\bm{\rT}}(v),\, \bm{\sigma}(n)\le \bm{\sigma}(v)\le n\}\\
=&\{p_{\rT_{n-1}}(v),\, v\in 1<\sigma_{n-1}(v)<K \}\\
& \cup \{p_{\bm{\rT}}(v),\, (2\vee K) \le \bm{\sigma}(v)\le n\},
\end{align*}
where the last two sets are conditionally independent given $\bm{\sigma}$. Now, since $(\rT_{n-1},\sigma_{n-1})$ is uniformly random in $\RD_{n-1}$,
the parents $\{p_{\rT_{n-1}}(v),\, v\in 1<\sigma_{n-1}(v)<K\}$ are independent, conditionally given $\sigma_{n-1}$ (and thus, also conditionally given $\bm{\sigma}$). On the other hand, for $v$ with $\bm{\sigma}(v)\ge K$,
\begin{align*}
p_{\bm{\rT}}(v)=\begin{cases}
n & \text{ if } X_{\bm{\sigma}(v)-1}=1, \\
p_{\rT_{n-1}}(v) &  \text{ if } X_{\bm{\sigma}(v)}=0, \\
\bm{\sigma}^{-1}(L) & \text{ if } \bm{\sigma}(v)=K.
\end{cases}
\end{align*}
Note that $p_{\bm{\rT}}(v)$ is determined independently from other vertices, thus $\{p_{\bm{\rT}}(v),\, K\le \bm{\sigma}(v)\le n\}$ are also independent, conditionally given $\bm{\sigma}$.
This implies that condition $ii)$  is satisfied.

Third, fix $1\le i<j\le n$ and fix distinct $v,w\in [n]$. We consider three cases; namely $v=n$, $w=n$, and $\{v,w\}\subset [n-1]$. Let 
\begin{align*}
A_1&=\{p_{\bm{\rT}}(n)=w,\, \bm{\sigma}(n)=j,\, \bm{\sigma}(w)=i\},\\
A_2&=\{p_{\bm{\rT}}(v)=n,\, \bm{\sigma}(v)=j,\, \bm{\sigma}(n)=i\}, \\
A_3&=\{p_{\bm{\rT}}(v)=w,\, \bm{\sigma}(v)=j,\, \bm{\sigma}(w)=i\}.
\end{align*}
It remains to show that the probabilities of $A_1,A_2,A_3$ are given by \eqref{uniform} for all $i,j\in [n]$. 
The event $p_{\bm{\rT}}(n)=w$ implies that $\sigma_{n-1}(w)=L<K$. Therefore, 
$A_1$ occurs precisely when $K=j$, $L=i$, and $\sigma_{n-1}(w)=i$. Then,
\begin{align*}
\p{A_1}=\p{K=j, L=i}\p{\sigma_{n-1}(w)=i}=\frac{1}{n(j-1)(n-1)}. 
\end{align*}
Next, $p_{\bm{\rT}}(v)=n$ implies that $\sigma_{n-1}(v)\ge K$ and thus $\bm{\sigma}(v)=\sigma_{n-1}(v)+1$. It then follows that 
$A_2$ occurs when  $K=i$, $\sigma_{n-1}(v)=j-1$, and $X_{j-1}=1$. Therefore,
\begin{align*}
\p{A_2}=\p{K=i, X_{j-1}=1}\p{\sigma_{n-1}(v)=j-1} =\frac{1}{n(j-1)(n-1)}. 
\end{align*}
For the last case, since $u,v<n$, it follows that $K\notin \{i,j\}$. For each $k\in [n]\setminus \{i,j\}$ let
\[A_{3,k}= \{p_{\bm{\rT}}(v)=w,\, \bm{\sigma}(v)=j,\, \bm{\sigma}(w)=i,\, K=k\}.\] 
In computing the probabilities $\p{A_{3,k}}$ we use that $(\rT_{n-1},\sigma_{n-1})$ is uniformly random in $RD_{n-1}$. If $K>j$, then both $\sigma_{n-1}(v)=\bm{\sigma}(v)$ and $\sigma_{n-1}(w)=\bm{\sigma}(w)$; in addition, $p_{\bm{\rT}}(v)=w$ only if $p_{\rT_{n-1}}(v)=w$. Therefore, if $k>j$, then
\begin{align*}
\p{A_{3,k}}
&=\p{K=k}\p{p_{\rT_{n-1}}(v)=w,\, \sigma_{n-1}(v)=j,\, \sigma_{n-1}(w)=i}\\
&=\frac{1}{n(n-1)(n-2)(j-1)}.
\end{align*}
Similarly, if $K<j$, then $\sigma_{n-1}(v)=\bm{\sigma}(v)-1$, $\sigma_{n-1}(w)=\bm{\sigma}(w)-\I{K<i}$, and additionally $X_{j-1}=0$. It then follows that, if $k<j$, 
\begin{align*}
\p{A_{3,k}}
&=\p{K=k,\,X_{j-1}=0}\p{p_{\rT_{n-1}}(v)=w,\, \sigma_{n-1}(v)=j-1,\, \sigma_{n-1}(w)=i-\I{K<i}}\\
&=\frac{1}{n}\cdot \frac{j-2}{j-1} \cdot \frac{1}{(n-1)(n-2)(j-2)}.
\end{align*}

We have shown that $\p{A_{3,k}}$ is uniform for all $k\in [n]\setminus \{i,j\}$, and we get 
\begin{align*}
\p{A_3}= \sum_{k\neq i,j} \p{A_{3,k}}= \frac{1}{n(n-1)(j-1)}.
\end{align*}
Altogether, we have shown that condition $iii)$ is satisfied and so the proof is complete.
\end{proof}

\section{The Poisson approximation}\label{sec:overview}

Recall that $(\rT_n,\sigma_n)$ is a uniform decorated tree and that $\rT_n$ has the shape of a recursive tree. In fact, \refP{prop:RD-R} implies that the following distributional identity holds, for all $n\in \N$,
\begin{equation*}
(\rd_{\rT_n}(\bm{\sigma}_n^{-1}(i));\, i\in [n])\eqdist (\rd_{\RRT_n}(i);\, i\in [n]).
\end{equation*}
It follows that the distribution of $(Z_m^\n,\, m\ge 1)$ and $\Delta_n$ does not change if we redefine them as $Z_{m}^\n=\#\{v\in [n]:\, \rd_{\rT_n}(v)\ge m \}$ and $\Delta=\max\{\rd_{\rT_n}(v):\, v\in [n]\}$. However, the correlations in 
$(\rd_{\rT_n}(v);\, v\in [n])$ have a subtle difference in comparison with those in $(\rd_{\RRT_n}(i);\, i\in [n])$.
To see this, observe that  $(\rd_{\RRT_n}(i),\, i\in [n])$ is negative orthant dependent; for a definition see \cite{DubhashiRanjan98}. This fact can be proven by induction from the two-vertex case $(\rd_{\RRT_n}(i),\rd_{\RRT_n}(j))$, which, in turn, follows essentially from the negative orthant dependency of multinomial distributions, see e.g. \cite[Lemma 1]{DevroyeLu95}. As a consequence, for all $i,j\in [n]$, 
\begin{align}\label{orthant}
\p{\rd_{\RRT_n}(i)\ge m,\, \rd_{\RRT_n}(j)\ge m}\le \p{\rd_{\RRT_n}(i)\ge m}\p{\rd_{\RRT_n}(j)\ge m}.
\end{align}
On the other hand, the following proposition gives conditions on $m$ for the degrees in $\rT_n$ to have a pairwise `almost' negative correlation. 

\begin{prop}\label{prop:asymp-neg}
For any $c\in (0,2)$ there exists $\alpha=\alpha(c)>0$ such that uniformly for $m=m(n)<c\ln n$ and distinct $v,w\in [n]$,   
\begin{align}\label{eq:asymp-neg}
\p{\rd_{\rT_n}(v)\ge m,\, \rd_{\rT_n}(w)\ge m}\le \p{\rd_{\rT_n}(v)\ge m}\p{\rd_{\rT_n}(w)\ge m} + O(2^{-2m-\alpha\log n}).
\end{align}
Moreover, $\alpha<\frac{1}{4}(1-c+\sqrt{1+2c-c^2})<1$.
\end{prop}

We make precise the constraints on $\alpha$ as this is crucial to \refT{thm:normal}. A weaker version of \refP{prop:asymp-neg}, without explicit error bounds, was proved in \cite[Proposition 4.2]{AddarioEslava15}; a complete proof of \refP{prop:asymp-neg} appears in Appendix A. 

Although we do not claim the bounds in \refP{prop:asymp-neg} are optimal, it seems that the property in \eqref{orthant} is lost when randomizing the vertex labels of $\RRT_n$ to obtain $\rT_n$. The bound in \eqref{eq:asymp-neg} will be an important input to the Chen-Stein Method. 

Briefly explained, our application of the Chen-Stein method compares, in total variation distance, the sum $Z_m^\n$ with respect to a Poisson variable with mean $\E{Z_m^\n}$. The strength of the bounds depend on finding suitable couplings between $(\I{\rd_{\rT_n}(v)\ge m}; v\in [n])$ and conditional versions of such variables. More precisely, we use the pruning procedure to obtain $T_n$ and Fact~\ref{fact:pruning} describes $(\rd_{\rT_n}(i),\, i\in [n])$  in terms of the independent elements $(\rd_{\rT_{n-1}}(i),\, i\in [n-1])$ and $\rd_{\rT_n}(n)$. This allows us to analyze the conditional law of $(\rd_{\rT_n}(i),\, i\in [n-1])$ given $\{\rd_{\rT_n}(n)\ge m\}$ holds. 

Before going into further details we layout the necessary notation. Given probability measures $\mu$ and $\nu$, a coupling of $\mu$ and $\nu$ is a pair $(X,Y)$ of  random variables (either real or vector-valued) with $X\sim \mu$ and $Y\sim \nu$. 
Let $I=(I_a ,\,a\in \cA)$ be a collection of $\{0,1\}$-valued random variables. Let $\mu$ be the law of $W=\sum_{a\in \cA} I_a$ and for $a\in \cA$ let $\nu_a$ be the conditional law of $W$ given that $I_a=1$, so 
\[\nu_a(B)=\p{W_a\in B}=\p{W\in B\,|\, I_a=1}.\]

We use the Chen-Stein method stated below.

\begin{thm}[{\cite[Theorem 3.7]{Goldschmidt00}}]\label{thm:ChenStein}
Let $I=(I_a ,\,a\in \cA)$ be a collection of $\{0,1\}$-valued random variables  and let $W=\sum_{a\in \cA} I_a$. For each $a\in \cA$ fix a coupling $(W,W_a)$ of $\mu$ and $\nu_a$. Then with $\lambda=\E{W}$, we have
\[\dTV(W,\Poi{\lambda})\le \min\{\lambda^{-1},1\} \sum_{a\in \cA} \E{I_a} \E{|W-(W_a-1)|}.\]
\end{thm}

To apply \refT{thm:ChenStein} with as tight as possible bounds, one can exploit properties of the variables $I_a$ or construct couplings of $\mu$ and $\nu$ with specific properties.

\begin{cor}\label{cor:ChenStein}
Let $I=(I_a ,\,a\in \cA)$ be a collection of $\{0,1\}$-valued random variables and let $W=\sum_{a\in \cA} I_a$.
If the variables $I=(I_a ,\,a\in \cA)$ are exchangeable, then for any fixed $a\in \cA$ and coupling $(W,W_a)$ of $\mu$ and $\nu_a$, we have
\begin{align}\label{cor:Chen1}
\dTV(W,\Poi{\lambda})\le \E{|W-(W_a-1)|}.
\end{align}
If, moreover, $W_a=(J_{ab},\, b\in \cA)$ and there is a coupling $(W,W_a)$ of $\mu$ and $\nu_a$ satisfying $J_{ab}\le I_a$ for all $b\in \cA\setminus \{a\}$, then
\begin{align}\label{cor:Chen2}
\dTV(W,\Poi{\lambda})\le \E{I_a}+\sum_{b\in \cA\setminus \{a\}} \E{I_a-J_{ab}}.
\end{align}
\end{cor}

Now, for the remainder of the section, fix $m$ and let, for all $v\in [n]$, $I_v=\I{\rd_{\rT_n}(v)\ge m}$, so that $Z_{m}^\n\eqdist \sum_{v\in [n]} I_v$. Let $(I,J)=((I_v,\, v\in [n]),(J_v,\, v\in [n])$ be a coupling of $\mu$ and $\nu=\nu_n$ where  $\mu$ is the law of $(I_1,\ldots,I_n)$ and $\nu=\nu_n$ is the conditional law of $(I_1,\ldots,I_n)$ given that $I_n=1$. 

If we would have orthant negative correlation for $(\rd_{\rT_n}(v),\, v\in [n])$ then it would follow that for all $v\in [n-1]$, $\E{I_nI_v}-\E{I_n}\E{I_v}\le 0$ and so the conditions for \eqref{cor:Chen2} would be satisfied. Although such strong property has not been yet established, \refP{prop:asymp-neg} implies for each $v\in [n-1]$, 
\begin{gather*}
\E{I_nI_v}-\E{I_n}\E{I_v}\le O(2^{-2m-\alpha\log n}).
\end{gather*} 
This suggests that there are couplings of $\mu$ and $\nu$ for which, with high probability, $I_v\le J_v$ for all $v\in [n-1]$. The existence for such couplings is delicate as the inequality $I_v\le J_v$ has to hold for all $v\in [n-1]$ simultaneously. 

The next proposition is the key ingredient in applying the Chen-Stein method to prove \refT{thm:dTV}. The coupling is based on the Robin-Hood pruning and its proof is the content of \refS{sec:DegCrux}. 

\begin{prop}\label{prop:coup-neg}
Let $c\in (1,2)$. There is $\beta=\beta(c)> 0$ such that for any $m=m(n)>c\ln n$ there exists a coupling $(I,J)=((I_1,\ldots, I_n),(J_1,\ldots, J_n))$ of $\mu$ and $\nu$, in which for all $v\in [n-1]$, 
\[\p{I_v< J_v}\le O(n^{-1-\beta}).\] 
\end{prop}

In the next section we assume \refP{prop:coup-neg} and complete the proofs of the results on high-degree vertices of $\RRT_n$. 

\subsection{Proofs for high-degree vertices}\label{sec:DegProofs}

\begin{proof}[Proof of \refT{thm:dTV}]
Fix $1<c'<c<2$ and let $c'\ln n<m=m(n)<c\ln n$. We apply the Chen-Stein method to $Z_m^\n\eqdist \sum_{v\in [n]} I_v$. 
First, we use the coupling $(I,J)=((I_1,\ldots, I_n),(J_1,\ldots, J_n))$ of $\mu$ and $\nu$ given in \refP{prop:coup-neg}. By \eqref{cor:Chen1}, we have 
\begin{align*}
\dTV \left(Z_m^\n,\Poi{\E{\lambda_{n,m}}}\right)
 &\le \E{|W-(W_n-1)|} \le \E{I_n}+\sum_{v\in [n-1]} \E{|I_v-J_v|}. 
\end{align*}
It thus remains to show that the terms in the bound above are $O(2^{-m+(1-\alpha)\log n})+O(n^{-\beta})$, where $\alpha=\alpha(c)\in (0,1)$ and $\beta=\beta(c')>0$ are defined as in Propositions \refand{prop:asymp-neg}{prop:coup-neg} respectively. For any $v\in [n-1]$,
\begin{align*}
\E{I_n}\E{|J_v-I_v|}
=&\E{I_n}\E{I_v-J_v}+2\E{I_n}\E{(J_v-I_v)\I{I_v<J_v}}\\
= & (\E{I_n}\E{I_v}-\E{I_nI_v}) +2\E{I_n}\p{I_v<J_v}.
\end{align*}
The terms in the last line are bounded by \eqref{eq:asymp-neg} and \refP{prop:coup-neg}, respectively. Since \eqref{eq:gamma} gives $\E{I_n}=2^{-m}(1+o(1))$
we get
\begin{align*}
\sum_{v\neq n} \E{|I_v-J_v|}
&=(n-1)\left[ \frac{\E{I_n}\E{I_v}-\E{I_nI_v}}{\E{I_n}}+ 2\p{I_v<J_v}\right]\\
&=O(2^{-m +(1-\alpha)\log n})+O(n^{-\beta}).
\end{align*}
Finally, \eqref{eq:gamma} together with $\alpha<1$ also gives $\E{I_n}=O(2^{-m+(1-\alpha)\log n})$.
\end{proof}

\begin{proof}[Proof of \refT{thm:normal}]

Fix $c\in (1,\log e)$ and let $\alpha=\alpha(c)$ be as in \refT{thm:dTV}. Using the upper bound for $\alpha$ in \refP{prop:asymp-neg} and simple computations yield $(1-\alpha)\log e<c$. Thus, we can chose $c'\in ((1-\alpha)\log e,c)$. 
Let $m=m(n)$ be such that $c'\ln n<m<c\ln n$. By the choice of $c$ and $c'$, we have that, as $n\to \infty$, $(1-\alpha)\log n-m<0$; while \eqref{eq:gamma} implies
\[\E{Z_m^\n}=2^{-m+\log n}(1+o(1)) \to \infty.\] 
The result then follows by \refT{thm:dTV} and the central limit theorem of Poisson variables, see e.g. \cite[Exercise 3.4.4]{Durret96}.  

\vspace{-.5cm}
\end{proof}

\begin{proof}[Proof of \refT{cor:Delta}]

Recall that $\eps_n=\log n-\floor{\log n}$. Let $i=i(n)$ satisfy $0<i< \log e \ln \ln n-C$, where $C>0$ is a constant to be determined below, and note that $2^{i+\eps_n}\le 2^{i+1} <2^{-C+1}\ln n$. Let $m=\floor{\log n}-i$ and $Z\eqdist \Poi{\lambda_{m,n}}$. 

We have that $\{\Delta_n<\floor{\log n}-i\}$ if and only if $\{Z_m^\n=0\}$. Therefore,  
\begin{align}\label{D}
\p{\Delta_n<\floor{\log n}-i} =\p{Z_m^\n=0} 
&\le \p{Z=0}+ \dTV(Z_m^\n,Z).
\end{align}
We deal with the two terms on the right-hand side of \eqref{D} separately. First,  using the lower bound on $i$, there is a constant $c\in (\log e,2)$ such that for $n$ large enough, $m-i<c\ln n$. Therefore, \eqref{eq:gamma} gives $\gamma>0$ such that 
$\lambda_{n,m}=2^{i+\eps_n} + o(n^{-\gamma} \ln n)$.
Consequently, 
\begin{align*}
\p{Z=0}=\exp\left\{-\lambda_{n,m}\right\}=\exp\{-2^{i+\eps_n}\}(1+o(1)).
\end{align*}
For the second term in \eqref{D}, \refT{thm:dTV} gives $\alpha, \beta>0$ such that 
\[\dTV(Z_m^\n,Z)= O(2^{-m+(1-\alpha)\log n})+O(n^{-\beta}).\]
It remains to deal with these two error terms. Note that $\exp\{2^{i+\eps_n}\}\le \exp\{2^{-C+1}\ln n\}$. Therefore, if $C>1+\log(1/\beta)$ then
\[\exp\{2^{i+\eps_n}\}O(n^{-\beta})=O(\exp\{(2^{-C+1}-\beta)\ln n\})\to 0;\] 
similarly, for $C$ large enough, 
\[\exp\{2^{i+\eps_n}\}O(2^{-m+(1-\alpha) \log n})=\exp\{2^{i+\eps_n}\}O(2^{i-\alpha \log n})\to 0.\] 
The two limits above imply that $\dTV(Z_m^\n,Z)=o(\exp\{-2^{i+\eps_n}\})$, completing the proof. 

\vspace{-.5cm}
\end{proof}

\section{The coupling for the Chen-Stein Method}\label{sec:DegCrux}

In this section we define and analyze the auxiliary coupling used in \refP{prop:coup-neg}. The coupling is based on the following straightforward property of the deterministic pruning. 
\begin{fact}\label{fact:pruning}
Fix $n\ge 2$. For $\rh_n((T,\sigma),(k,l,x))=(T',\sigma')$, we have
$\rd_{T'}(n) = \sum_{i=k}^{n-1} x_i$,
and for $v\in [n-1]$,
\begin{align*}
\rd_{T'}(v)& = \rd_{T}(v)+\I{l=\sigma(v)} -\sum_{i=k}^{n-1} x_i\I{v=p_T(\sigma^{-1}(i))}.
\end{align*}
\end{fact}
In words, Fact~\ref{fact:pruning} specifies when the degree of a vertex $v<n$ changes: either for having $n$ as a new child or for losing children that are rewired towards $n$. Clearly, the degree of $n$ equals the total number of such rewirings. 

The heuristic for the \emph{almost} negative relation obtained in \refP{prop:coup-neg} is the following. Start with $(\rT_{n-1},\sigma_{n-1})$ and apply the Robin-Hood procedure. If the degree of vertex $n$ is large, Fact~\ref{fact:pruning} implies that a large number of vertices in $\rT_{n-1}$ were rewired towards $n$ in the new tree; thus, many (parent) vertices decreased their degree by at least one. In short, conditioning on $\deg_{\rT_n}(n)\ge m$ implies that other vertices are (slightly) less likely to satisfy $\deg_{\rT_n}(v)\ge m$. 

For the remainder of the section, fix $n\in \N$, $c\in (1,2)$ and $m=m(n)>c\ln n$. Let $(\rT_{n-1}, \sigma_{n-1})$ be uniformly random in $\RD_{n-1}$, $(K,L,X)$ be an $\RH_{n}$-set, and $(K',L',X')$ be distributed as an $\RH_{n}$-set conditioned to satisfy $\sum_{i=K}^{n-1} X'_i\ge m$. Now, write
\begin{align}
(\rT_n,\sigma_n)&=\rh((\rT_{n-1},\sigma_{n-1}),(K,L,X)), \label{T}\\
(\bm{\rT},\bm{\sigma})&=\rh((\rT_{n-1},\sigma_{n-1}),(K',L',X')). \label{Tm}
\end{align}
To avoid cluttery notation, we omit the dependency on $m$ of the conditional random variables $(K',L',X')$ and $(\bm{\rT},\bm{\sigma})$. 
By Fact~\ref{fact:pruning} and \refT{thm:main}, $(\bm{\rT},\bm{\sigma})$ is a conditional version of  $(\rT_{n},\sigma_{n})$ given that $\rd_{\rT_{n}}(n)\ge m$. 
Consequently, if $I_v=\I{\rd_{\rT_n}(v)\ge m}$ and $J_v=\I{\rd_{\bm{\rT}}(v)\ge m}$ for all $v\in [n]$, then any coupling between $(K,L,X)$ and $(K',L',X')$ yields a coupling for the measures $\mu$ and $\nu$ in \refP{prop:coup-neg}. 

Our goal is then to couple $(K,L,X)$ and $(K',L',X')$ in such a way that the negative relation between $I_v$ and $J_v$ fails on a negligible set. 
More precisely, we construct a coupling so that there is $\beta=\beta(c)>0$ satisfying 
\begin{align}\label{coupling}
\p{I_v<J_v}=\p{\rd_{\rT_{n-1}}(v)< m\le \rd_{\bm{\rT}}(v)}=O(n^{-1-\beta}). 
\end{align}
 
Lemmas~\ref{lem:StrassenX}--\ref{lem:coup1} provide the coupling between $(K,L,X)$ and $(K',L',X')$, while \refP{prop:coup1} gives necessary conditions, under the coupling, for $I_v<J_v$ to hold. The proof of \refP{prop:coup-neg} then follows from bounding the probability that such necessary conditions occur.

\subsection{Construction of the coupling}

For any integer $n-m\le k<n$, let $X^k=(X_i^k,\, i\in [n-1])$ be a conditional version of $X$ given that $\sum_{i=k}^{n-1} X_i\ge m$. The following observation is quite standard but we include a proof for completeness. For $a=(a_1,\ldots, a_{d})$ and $b=(b_1, \ldots, b_{d})\in \{0,1\}^{d}$, $a\le b$ only if $a_i\le b_i$ for all $i\in [d]$. We say that $S\subset \{0,1\}^d$ is monotone if $a\le b$ and $a\in S$ imply $b\in S$.

\begin{lem}\label{lem:StrassenX}
For each $k<n$, there exists a coupling of $X^k$ and $X$ such that $X_i\le X^k_i$ for all $i\in [n-1]$. 
\end{lem}

\begin{proof}
Fix $k<n$. Note that $S_k=\{a \in \{0,1\}^{n-1}: a_k+\ldots +a_{n-1}\ge m\}$
is a monotone subset of $\{0,1\}^{n-1}$. Harris inequality implies
$\p{X\in S\cap S_k}\ge \p{X\in S_k}\p{X\in S}$, for any monotone subset $S\in \{0,1\}^{n-1}$.
Dividing through by $\p{X\in S_k}$ yields $\p{X^k \in S}\ge \p{X\in S}$. Therefore, $X^k$ stochastically dominates $X$. The existence of the coupling is then guaranteed by Strassen's theorem \cite{StrassenExp}.
\end{proof}

Before the next coupling, we gather two observations. First, for fixed $(k,l)$, we have $\p{L=l|K=k}=\p{L'=l|K'=k}$. To see this, observe that $\p{L'=l|K'=k}$ can be rewritten as 
\begin{align*}
\frac{\p{L=l,\,K=k,\, \sum_{i=K}^{n-1}X_i\ge m}}{\p{K=k,\, \sum_{i=K}^{n-1}X_i\ge m}}
=\frac{\p{L=l,\,K=k,\, \sum_{i=k}^{n-1}X_i\ge m}}{\p{K=k,\, \sum_{i=k}^{n-1}X_i\ge m}};
\end{align*}
the claim then follows by the independence between $X$ and $(K,L)$.
Second, the sequence 
$p_k=\p{K=k\,|\, \sum_{i=K}^{n-1} X_i\ge m}$ is proportional to $\p{\sum_{i=k}^{n-1} X_i\ge m}$ and, thus, it is decreasing in $k$. Clearly, the latter sequence of probabilities is decreasing in $k$, while both are proportional with a factor
$Z=n\p{\sum_{i=K}^{n-1}X_i\ge m}$. To see this, use the independence between $X$ and $K$ to obtain
\begin{align*}
\p{K=k\,\middle |\, \sum_{i=K}^{n-1} X_i\ge m}
&=\frac{\p{K=k,\sum_{i=k}^{n-1}X_i\ge m}}{\p{\sum_{i=K}^{n-1}X_i\ge m}}
=\frac{1}{Z}\,\p{\sum_{i=k}^{n-1}X_i\ge m}.
\end{align*}

\begin{lem}\label{lem:StrassenM}
There exists a coupling of $(K,L)$ and $(K',L')$ such that $K'\le K$ and $L'\le L$. 
\end{lem}

\begin{proof}
Let $X=(X_1,\ldots, X_{n-1})$ be independent with $X_i\eqdist \Ber{1/i}$ and independently, let $U_1,U_2$ be i.i.d. $\Unif{0,1}$. By a slight abuse of notation we \emph{redefine} the variables $(K,L)$ and $(K',L')$ using the variables $U_1,U_2$ and argue that the \emph{original} law is preserved. 

Let $(K,L)=(\ceil{nU_1},\ceil{(K-1)U_2})$ and $(K',L')=(K',\ceil{(K'-1)U_2})$ with \[K'=\max\left\{k:\, U_1 > \sum_{j=1}^{k-1} p_j\right\}.\]
It is straightforward that $(K,L)$ and $K'$ have the correct law by construction, while $L'$ has the correct law since $\p{L=l|K=k}=\p{L'=l|K'=k}$ for each $0\le l<k\le n$. 
Moreover, since $p_k$ is decreasing, it follows that $K'=j$ implies
$U_1> \sum_{i=1}^{j-1} p_i\ge \frac{j-1}{n}$. 
It follows that $K\ge j=K'$, and so $L=\ceil{(K-1)U_2} \ge \ceil{(K'-1)U_2} =L'$. \end{proof}

\begin{lem}\label{lem:coup1}
There exists a coupling of $(K,L,X)$ and $(K',L',X')$ such that $K'\le K$, $L'\le L'$ and $X_i\le X'_i$ for all $i\in [n-1]$.
\end{lem}

\begin{proof}
Let $U_1,U_2$ be i.i.d. $\Unif{0,1}$ and independently, let $X=(X_1,\ldots, X_{n-1})$ be independent with $X_i\eqdist \Ber{1/i}$. For each $1\le k<n$ fix a vector $X^k$ coupled with $X$ according to \refL{lem:StrassenX}. 
The dependence structure of $X^1, \ldots, X^{n-1}$ is unimportant to the argument, but for concreteness we may, e.g., take them to be conditionally independent given $X$. On the other hand, it is important to insist that the $X^k$ are independent of $(K',L')$. Since we will define $(K',L')$ using $U_1,U_2$, the existence of such joint coupling is straightforward.

Again, by a slight abuse of notation we \emph{redefine} the variables and argue that the \emph{original} law is preserved. Define $(K,L)$,$(K',L')$ as in \refL{lem:StrassenM} and let $X'=X^{K'}$. Clearly, $(K,L,X)$ is an $\RH_n$-set. It remains to show that $(K',L',X')$ has the conditional distribution of $(K,L,X)$ given that $\sum_{i=K}^{n-1}X_i\ge m$.
For any $(k,l,x)\in \HP_{n}$, the probability $\p{(K,L,X)=(k,l,x)\,\middle |\, \sum_{i=K}^{n-1} x_i\ge m}$ can be rewritten as
\begin{align*}
\frac{\p{K=k,\,L=l,\,X=x,\, \sum_{i=k}^{n-1} x_i\ge m}}{\p{\sum_{i=K}^{n-1} x_i\ge m}}
&=\frac{\p{K=k,\,L=l}\p{X=x,\,\sum_{i=k}^{n-1} x_i\ge m}}{\p{\sum_{i=K}^{n-1} x_i\ge m}}.
\end{align*}
Adding two factors of $\p{\sum_{i=k}^{n-1} x_i\ge m}$ and using the independence between $(K,L)$ and $X$, we can factorize these probabilities as
\begin{align*}
\frac{\p{K=k,\,L=l,\,\sum_{i=k}^{n-1} x_i\ge m}}{\p{\sum_{i=K}^{n-1} x_i\ge m}}
\cdot
\frac{\p{X=x,\,\sum_{i=k}^{n-1} x_i\ge m}}{\p{\sum_{i=k}^{n-1} x_i\ge m}}
\end{align*}
These probabilities correspond, respectively, to the distributions of $(K',L')$ and $X^k$, which are independent. Therefore,
\begin{align*}
\p{(K,L,X)=(k,l,x)\,\middle |\, \sum_{i=K}^{n-1} x_i\ge m}
&=\p{(K',L')=(k,l)}\p{X^k=x}\\
&=\p{(K',L',X')=(k,l,x)}
\end{align*}
as desired. Finally, the variables $(K,L,X)$ and $(K',L',X')$ satisfy the desired inequalities by Lemmas \refand{lem:StrassenX}{lem:StrassenM}.
\end{proof}

\subsection{Analysis of the coupling}
The proof of \refP{prop:coup-neg} boils down to understanding necessary conditions for $\rd_{\rT_n}(v)< m\le \rd_{\bm{\rT}}(v)$ to hold under the coupling of \refL{lem:coup1}.

\begin{prop}\label{prop:coup1}
Consider $(K,L,X)$ and $(K',L',X')$ defined in \refL{lem:coup1} and their corresponding decorated trees $(\rT_n,\sigma_n),(\bm{\rT},\bm{\sigma})$ defined in \eqref{T} and \eqref{Tm}. For any $v\in [n-1]$, 
\[\{\rd_{\rT_n}(v)< m\le \rd_{\bm{\rT}}(v)\}\subset \{L'=\sigma_{n-1}(v)\}\cap \{\rd_{\rT_{n-1}}(v)\ge m-1\}.\] 
\end{prop}

\begin{proof}
From the properties of the coupling in \refL{lem:coup1},
\begin{align}\label{Xs}
\sum_{i=K}^{n-1} X_i\,\I{v=p_{\rT_{n-1}}(\sigma_{n-1}^{-1}(i))}
\le \sum_{i=K'}^{n-1} X'_i\,\I{v=p_{\rT_{n-1}}(\sigma_{n-1}^{-1}(i))}.
\end{align}
Consequently, using Fact \ref{fact:pruning} we have that  
$\rd_{\bm{\rT}}(v)-\rd_{\rT_n}(v)\le \I{L'=\sigma_{n-1}(v)}$.
On the other hand, if $\{\rd_{\rT_n}(v)< m\le \rd_{\bm{\rT}}(v)\}$ holds, then it follows that $\rd_{\bm{\rT}}(v)-\rd_{\rT_n}(v)>0$ and so it is necessary that $\{L'=\sigma_{n-1}(v)\}$ holds. 
Finally, $\{m\le \rd_{\bm{\rT}}(v)\}$ implies that 
\[m\le \rd_{\bm{\rT}}(v)= \rd_{\rT_{n-1}}(v) +\I{L'=\sigma_{n-1}(v)}- \sum_{i=K'}^{n-1} X'_i\,\I{v=p_{\rT_{n-1}}(\sigma_{n-1}^{-1}(i))}
\le \rd_{\rT_{n-1}}(v) +1;\]
or equivalently, that  $\{\rd_{\rT_{n-1}}(v)\ge m-1\}$.
\end{proof}

We can also argue, more specifically, that 
\[\{\rd_{\rT_n}(v)< m\le \rd_{\bm{\rT}}(v)\}\subset \{L'=\sigma_{n-1}(v)\}\cap \{\rd_{\rT_{n-1}}(v)= m-1\};\]
 however, the approach we chose allow us to use uniform bounds for all $v\in [n-1]$. We will frame the events $\{\rd_{\rT_{n-1}}(v)\ge m-1\}$ from the perspective of recursive trees where the degree distributions are distinct for each vertex.  Recall the following version of Bernstein inequalities (see, e.g. \cite{JLR} Theorem~2.8, (2.5)). For a sum $S$ of $\{0,1\}$-valued variables and $\eps>0$, 
$\p{S>(1+\eps)\E{S}}\le \exp\left\{-\frac{3\eps^2}{2(3+\eps)}\E{S}\right\}$.
By the construction of $\RRT_n$ we have that 
$\rd_{\RRT_n}(i)\eqdist \sum_{k=i}^n B_k \le \sum_{k=1}^n B_k$ where $(B_k,\, k\ge 1)$ are independent Bernoulli variables with mean $1/k$. Therefore, 
\begin{align*}
\p{\rd_{\RRT_{n}}(i)> m}\le \p{\sum_{k=1}^{n} B_k> c\ln n}.
\end{align*}
Using that $\E{\sum_{k=1}^{n} B_k}=\ln n+O(1)<c\ln n$, we can apply Berstein's inequality with $\eps=c-1+o(1)$ and set $\beta=\frac{3\eps^2}{2(3+\eps)}$. 
It follows that there is $\beta=\beta(c)>0$ such that uniformly over $m>c\ln n$, and $i\in [n]$,  
\begin{equation}\label{eq:tailRRT}
\p{\rd_{\RRT_n}(i)>m}=O(n^{-\beta}).
\end{equation}

\begin{proof}[Proof of \refP{prop:coup-neg}]
Fix $c\in (1,2)$. Let $m=m(n)>c\ln n$ and $\beta=\beta(c)>0$ be as in \eqref{eq:tailRRT}. Let $\rT_n$ and $\bm{\rT}$ be as defined in \eqref{T} and \eqref{Tm} with $((K,L,X),(K',L',X'))$ as in \refL{lem:coup1}. Set $I_v=\I{\rd_{\rT_n}(v)\ge m}$ and $J_v=\I{\rd_{\bm{\rT}}(v)\ge m}$ for all $v\in [n]$, so that $(I,J)=((I_1,\ldots, I_n),(J_1,\ldots J_n))$ is a coupling of the measures $\mu$ and $\nu$. 

Our goal is to bound $\p{I_v<J_v}=\p{\rd_{\rT_n}(v)< m\le \rd_{\bm{\rT}}(v)}$. First, by \refP{prop:coup1}, 
\begin{align*}
\p{\rd_{\rT_n}(v)< m\le \rd_{\bm{\rT}}(v)}
&\le \sum_{j=1}^{n-1} \p{L'=j,\, \sigma_{n-1}(v)=j,\, \rd_{\rT_{n-1}}(v)\ge m-1}.
\end{align*}
Next we obtain uniform bounds for the terms on the right-hand side. Recall that $\sigma_{n-1}$ is a uniformly random permutation independent of $L'$ and that $\sigma_{n-1}(\rT_{n-1})\eqdist \RRT_{n-1}$. These facts, together with \eqref{eq:tailRRT} gives, for each $j\in [n-1]$,
\begin{align*}
\p{L'=j,\, \sigma_{n-1}(v)=j,\, \rd_{\rT_{n-1}}(v)\ge m-1}
&= \frac{\p{L'=j}}{n-1} \p{\rd_{\rT_{n_1}}(v)\ge m-1\,|\,\sigma_{n-1}(v)=j}\\
&= \frac{\p{L'=j}}{n-1} \p{\rd_{\RRT_{n-1}}(j)\ge m-1}\\
& \le \p{L'=j} O(n^{-1-\beta}).
\end{align*}
Plugging together these bounds, we get for any $v\in [n-1]$,  
\begin{align*}
\p{I_v<J_v}
&\le \sum_{j=1}^{n-1} \p{L'=j,\, \sigma_{n-1}(v)=j,\, \rd_{\rT_{n-1}}(v)\ge m-1}\\
&= O(n^{-1-\beta}) \sum_{j=1}^{n-1} \p{L'=j}=O(n^{-1-\beta}).
\end{align*}

\vspace{-.5cm}
\end{proof}

\section{Conclusions and further research} \label{sec:Conclusions}

The Robin-Hood pruning yields an interesting process $((\rT_n, {\sigma}_n),\, n\ge 1)$. By \refT{thm:main} and \refP{prop:RD-R}, $\sigma_n(\rT_n)\eqdist \RRT_n$ for all $n\ge 1$; that is, $\rT_n$ has the shape of a recursive tree. The novelty of this process is that the Robin-Hood pruning is a fairly complex dynamic of trees which has potential connections to mathematical models of social and economic networks and raises challenging theoretical questions. 
 
First, only asymptotically about half the time $\rT_n$ is obtained from $\rT_{n-1}$ by simply attaching $n$ to a uniformly random vertex. To see this,  recall that $\rd_{\rT_n}(n)\eqdist \min\{\Geo{1/2},|\cS|\}$ where $|\cS|\to \infty$ (see Fact~\ref{deg id} and \refL{S id}). It follows that with probability tending to $1/2$ the newly added vertex will be a leaf. Second, for all $n\ge 1$, $\rd_{\RRT_n}(n)=0$ a.s.~, while 
Fact \ref{fact:pruning} and the distribution of the $\RH$-set yield
\begin{align*}
\E{\rd_{\rT_n}(n)}=\E{\E{\textstyle{\sum_{i=k}^{n-1} X_i}\,|\, M=k}}
&=\sum_{k=1}^n\sum_{i=k}^{n-1} \frac{1}{n\cdot i}
=\sum_{i=1}^{n-1}\sum_{k=1}^{i} \frac{1}{n\cdot i}=1-\frac{1}{n}.
\end{align*}
Third, from time to time, a large proportion of edges will be rewired towards the newly added vertex, drastically reshaping the structure of the tree. 
For example, for any $a\in [0,1)$, 
\[\E{\rd_{\rT_n}(n)\,|\, M\le n^a}\ge \E{\sum_{i=n^a}^n X_i}=(1-a)\ln n.\] 

As for applications, in the context of random networks, the Robin-Hood pruning has an interpretation in terms of `trends'; for example, a new vertex brings in a new idea to the network which may drastically rewire the interests or connections of established individuals in the network. The stamp history $\sigma_n$ gives a ranking between the elements of $\rT_n$ that determines the susceptibility of changing parents in the tree. Preferential attachment models are considered better models for real-world networks. It would be interesting to devise a similar pruning procedure that, acting on preferential attachment trees, preserves their scale-free degree distribution. 

In the context of biology, Kingman's coalescent is usually represented with increasing binary trees, keeping individuals as external nodes and \emph{adding} an internal node for each merge between two lineages. The representation using $n$-chains breaks the symmetry between the pairs of trees merging at each step. Thus, it is not clear how the Robin-Hood pruning process would have a significant interpretation in terms of the genealogical information. 

Regardless of the perspective we use to motivate the process $((\rT_n, \sigma_n),\, n\ge 1)$, there are many interesting theoretical questions that would be worth pursuing. To name just a few:
\begin{enumerate}
\item Understand the process describing how the parent and descendants of a given vertex change with time.
\begin{itemize}
\item Describe how the size of the subtree rooted at a fixed node $j$ evolves. 
\item How does maximum size of such subtree grow?
\end{itemize} 
\item Understand the maximum degree dynamics in both $(R_n,\, n \ge 1)$ and $(\rT_n,\, n \ge 1)$.
\begin{itemize}
\item How often does vertices attaining the maximum degree change? 
\item Are this dynamics the same for both processes?
\end{itemize}
\item Determine whether there is a coupling for which the sequence $(\I{\rd_{\rT_n}(v)},v\in [n])$ is negative related (i.e. that conditions for \eqref{cor:Chen2} are satisfied), or similarly, whether the sequence is negative orthant dependent. 
\end{enumerate}

\section*{Acknowledgements}

I would like to thank Louigi Addario-Berry and Henning Sulzbach for some very helpful discussions, and to the anonymous referees who provided insight on how to improve the presentation of the results and additional references. This research was supported by FQRNT through PBEEE scholarship with number 169888.

\section*{Appendix A: Proof of \refP{prop:asymp-neg}}\label{sec:Appendix}

We use the representation of Kingman's coalescent that consists of a chain $\mathbf{C}=(F_n,\ldots, F_1)$ and write $T^\n$ for the unique tree contained in $F_1$. By \refP{prop:K-DT} we can work with the tree $T^\n$. The proof mimics that of \cite[Proposition 4.2]{AddarioEslava15}, but requires a little more care as we wish to obtain explicit error bounds. 

For each $v,j\in [n]$ let $T_j(v)$ denote the tree in $F_j$ that contains vertex $v$. For each $v\in [n]$, the \emph{selection set} of $v$ is defined as
\[\cS_n(v)=\{2\le j\le n:\, T_j(v)\in \{T_{a_j}^{(j)},T_{b_j}^{(j)}\} \}; \]
this set keeps record of the times when the tree containing $v$ merges.
Finally, for each $2\le j\le n$, we say that $\xi_j$ is \emph{favorable} for vertices in $T_{a_j}^{(j)}$ (resp. vertices in $T_{b_j}^{(j)}$) if $\xi_j=1$ (resp. $\xi_j=0$). 

The key property of Kingman's coalescent is the following. For each $j\in \cS_n(v)$, if $\xi_j$ favors $v$, then $r(T_j(v))$ increases its degree by one in the process; otherwise $r(T_j(v))$ attaches to the root of the other merging tree and the degree of $r(T_j(v))$ remains unchanged for the rest of the process. 
Since all vertices start the process as roots, $\rd_{T^\n}(v)$ is equal to the length of the first streak of favorable times for $v$. Moreover, $(\xi_j,\, j\in [n-1])$ are independent and distributed as $\Ber{1/2}$. Therefore we have the following distributional equivalence.

\begin{fact}\label{deg id}
Let $D$ be a random variable with distribution $\Geo{1/2}$ independent of $\cS_n(v)$, then
\[\rd_{T^\n}(v)\eqdist \min\{D,|\cS_n(v)|\}.\]
\end{fact}

This fact, together with the next lemma, allow us to get estimates for the tails of $\rd_{T^\n}(v)$. 

\begin{lem}\label{S id}
If $c\in (0,2)$ and $0<\eps\le 1-c/2$. Writing $a=1-\eps-c/2$, we have 
\[\p{|\cS_n(v)\setminus [n^a]|>c\ln n}\le O(1)n^{-\eps^2/(\eps+c/2)}.\]
\end{lem}

\begin{proof}
First, there are $j(j-1)$ distinct pair of trees in $F_j$, exactly $j-1$ of such pairs contains $T_j(v)$; thus $\p{j\in \cS_n(v)}=2/j$. Since the merging trees are chosen independently at each time, we have that for any $a\in [0,1)$ we have
\begin{align*}
|\cS_n(v)\setminus [n^a]|\eqdist \sum_{j=n^a+1}^n B_j, 
\end{align*}
where the variables $B_1,\ldots B_n$ are independent Bernoulli variables with $\E{B_i}=2/i$, respectively. 
The desired bound is then a straightforward application of Bernstein's inequalities (see, e.g. \cite{JLR}, Theorem~2.8 and (2.6)). For a sum $S$ of $\{0,1\}$-valued variables, we have $\p{S\le \E{S}-t} \le \exp\{-t^2/2\E{S}\}$. 
In this case, $S=\sum_{i=n^a}^n B_i$ and 
\[\E{S}=\sum_{i=n^a}^n \frac{2}{i}=2(1-a)\ln n+O(1)=(c+2\eps)\ln n+O(1).\] 
The result follows by setting $t=2\eps\ln n+O(1)$. 
\end{proof}

\begin{prop}\label{upper m}
If $c\in (0,2)$ and $m<c\ln n$, then for $\eps=(2-c)^2/4$,
\[2^{-m}(1-o(n^{-\eps}))\le \p{\rd_{T^\n}(1)\ge m}\le 2^{-m}.\]
\end{prop}

\begin{proof}[Proof of \refP{upper m}]
It follows from \refL{deg id} that 
\[\p{\rd_{T^\n}(v)\ge m} =\p{D\ge m} \p{|\cS_n(v)|\ge m}.\]
The upper bound on $\p{\rd_{T^\n}(1)\ge m}$ is then trivial, while the lower bound follows by \refL{S id} using $\eps=1-c/2$ and that $\cS_n(v)=\cS_n(v)\setminus [1]$. 
\end{proof}

Now, consider two distinct vertices $v,w\in [n]$. For $m\in \N$, let $\cG_m\in \{2,\ldots, n\}^2$ contain all pairs of selection sets that enable vertices $v$ and $w$ to have degree at least $m$; that is, $(A,B)\in \cG_m$ only if  
\[\p{\rd_{T^\n}(v)\ge m,\, \rd_{T^\n}(w)\ge m, (\cS_n(v),\cS_n(w))=(A,B)}>0.\]
Since the $\xi_j$ are independent of the selection times, we have that 
\begin{align}\label{G}
\p{\rd_{T^\n}(v)\ge m,\, \rd_{T^\n}(w)\ge m}\ge 2^{-2m}\p{(\cS_n(v),\cS_n(w))\in \cG_m}.
\end{align}

To estimate $\p{(\cS_n(v),\cS_n(w))\in \cG_m}$ we need more details on the dynamics of the model. We start with a simple tail bound for the following random variable; let 
\[\tau=\max\{j:\, j\in \cS_n(v)\cap \cS_n(w)\}.\]

\begin{lem}\label{tau}
For $a\in (0,1)$, $\p{\tau>n^a}\le 4n^{-a}$.
\end{lem}

\begin{proof}
Vertices in $T^\n$ are exchangeable, so we can take $v=1,w=2$; these vertices belong to distinct trees in $F_j$ for all $j\ge \tau$.  Additionally, by the ordering convention of trees in $F_j$, it follows that $T_j(1)=1$ and $T_j(2)=2$ for all $j\ge \tau$. 

We claim that for all $2<k\le n$, 
\[\p{\tau \le k}=\prod_{j=k+1}^n \left(1-\frac{2}{j(j-1)}\right).\]
This follows by induction on $n-k$. Clearly, $\tau=n$ only if $\{a_n,b_n\}=\{1,2\}$ which occurs with probability $\frac{2}{n(n-1)}$, thus $\p{\tau\le n-1}$ satisfies the equation above. For $k<n$, we have 
\begin{align*}
\frac{\p{\tau\le k}}{\p{\tau \le k+1}}=\p{\tau\le k |\tau\le k+1}
&=\p{\{a_{k+1},b_{k+1}\}\neq \{1,2\}}=1-\frac{2}{(k+1)k}.
\end{align*}
Next, for $k$ larger enough, 
\begin{align*}
\prod_{j=k+1}^n \left(1-\frac{2}{j(j-1)}\right)\ge \prod_{j=k}^{n-1} \left(1-\frac{2}{j^2}\right)>1-\sum_{j=k}^\infty \frac{2}{j^2} >1 -4 \int_{k}^{\infty} x^{-2}dx= 1-4/k.
\end{align*}
The second inequality uses that $1-x>e^{-2x}$ for $x>0$ sufficiently small, followed by the fact that $e^{-\sum 2x_j}>1-\sum 2x_j$. The result follows with $k=n^a$. 
\end{proof}

\begin{lem}\label{lower m}
If $c\in (0,2)$ and $m<c\ln n$, then for any $\gamma<\frac{1}{4}(1-c+\sqrt{1+2c-c^2})$,
\[\p{(\cS_n(v),\cS_n(w))\in \cG_{m}}\ge 1- o(n^{-\gamma}).\]
\end{lem}

\begin{proof}
For each $\eps\in (0,1-c/2]$ write $a=a(\eps)=1-\eps-c/2$, then 
\begin{align}\label{Ga}
\p{(\cS_n(v),\cS_n(w))\notin \cG_{m}}&\le \p{\tau>n^a} +2\p{|\cS_n(v)\setminus [n^a]|<c\ln n}.
\end{align}
Before, establishing \eqref{Ga}, we note that the terms in the right-hand side of \eqref{Ga} are bounded by Lemmas \refand{tau}{S id}, respectively. Since such bounds depend on the choice of $\eps$, we can use 
\[\gamma<\max_{0<\eps \le 1-c/2}
\left\{
\min\left(1-\eps-\frac{c}{2},\frac{\eps^2}{\eps+\frac{c}{2}}\right)
\right\}
=\frac{1}{4}\left(1-c+\sqrt{1+2c-c^2}\right).\]

The last equality since the functions to be minimized are decreasing and increasing, respectively, on the $(0,1)$ interval. It then follows that the maximum is attained when $0<\eps<1-c/2$ satisfies 
$1-\eps-c/2=\eps^2/(\eps+\frac{c}{2})$.

We now proceed to establish equation \eqref{Ga}. At step $\tau$, exactly one of $v$ and $w$ is favored by $\xi_\tau$. Thus, at least one of $v$ or $w$ gets its degree fixed for the remainder of the process. Therefore,
\begin{align*}
\left\{(\cS_n(v),\cS_n(w))\in \cG_m \right\}
\subset 
\left\{ |\cS_n(v)\setminus [\tau]|\ge m \right\}
\cup 
\left\{ |\cS_n(w)\setminus [\tau]|\ge m \right\}.
\end{align*}
By intersecting with the event $\tau> n^a$, and the exchangeability of vertices in $T^\n$ we get, 
\begin{align*}
\p{(\cS_n(v),\cS_n(w))\notin \cG_{m}}
&\le \p{\tau> n^a} + 2\p{(\cS_n(v),\cS_n(w))\notin \cG_{m}, \tau\le n^a}\\
&\le \p{\tau> n^a} + 2\p{|\cS_n(v)\setminus [\tau]|< m, \tau\le n^a}\\
&\le \p{\tau> n^a} + 2\p{|\cS_n(v)\setminus [n^a]|< m, \tau\le n^a};
\end{align*} 
from which \eqref{Ga} follows. 
\end{proof}

\begin{proof}[Proof of \refP{prop:asymp-neg}]
Fix $c\in (0,2)$, $m=m(n)<c\ln n$ and let $I_v,J_{v}$ be defined as in \refP{prop:asymp-neg}. By \refP{prop:K-DT}, if follows that $\E{I_v}=\p{\rd_{T^\n}(v)\ge m}$ and 
\begin{align*}
\E{I_v}\E{J_{v}}=\E{I_vI_n}
&=\p{\rd_{T^\n}(v)\ge m,\, \rd_{T^\n}(n)\ge m}\\
&=2^{-2m} \p{(\cS_n(v),\cS_n(n))\in \cG_{m}};
\end{align*}
the last equality by \eqref{G}. Lemmas \refand{lower m}{upper m} then gives that for $\alpha <\frac{1}{4}(1-c+\sqrt{1+2c-c^2})$,
\begin{align*}
\E{I_v}\E{I_n}-\E{I_v}\E{J_{vn}}\le 2^{-2m}-2^{-2m}(1+o(n^{-\alpha}))=2^{-2m}o(n^{-\alpha}).
\end{align*}

\vspace{-.5cm}
\end{proof}


\def\cprime{$'$}

\Addresses

\end{document}